\documentclass[10pt,a4paper,times]{article}
\usepackage[latin1,applemac]{inputenc}
\usepackage{amsmath, amssymb, amsfonts, amsthm, amscd, epsfig, multicol, float,mathrsfs,enumerate,,float}
\usepackage[left=2cm,top=3cm,bottom=3cm,right=2cm]{geometry}
\usepackage[scaled=.95]{helvet} 
\usepackage{courier}

\newtheorem{theorem}{Theorem}[section]
\newtheorem{lemma}[theorem]{Lemma}

\theoremstyle{definition}

\newtheorem{cor}[theorem]{Corollary}

\newtheorem{remark}[theorem]{Remark}

\numberwithin{equation}{section}
\newcommand{\re}[1]{\mbox{Re}(#1)}

\begin{document}

\title{On computing some special values of multivariate hypergeometric functions}

%    Remove any unused author tags.

%    author one information
\author{Giovanni Mingari Scarpello \footnote{giovannimingari@yahoo.it}
\and Daniele Ritelli \footnote{Dipartimento di Scienze Statistiche, via
Belle Arti, 41 40126 Bologna Italy, daniele.ritelli@unibo.it}}

%\subjclass[2010]{33C65, 33C07, 33E05}
%    For articles to be published after 1 January 2010, you may use
%    the following version:
%\subjclass[2010]{Primary }

%\keywords{}

\date{}

%\dedicatory{}

\maketitle

\begin{abstract}
This paper obtains several evaluations of multivariate hypergeometric functions for particular parameter values and at special algebraic points. They have a high interest not only on their own, but also in the light of the remarkable implications for both pure mathematics and  mathematical physics. Following our research started in \cite{jnt1} and \cite{jnt2}, we provide some contribution to such functions' computability inside and outside their disks of convergences.

In the first part we provide some new results in the spirit of theorem 3.1 of \cite{jnt2}, obtaining formulae for the values  of multivariate hypergeometric functions by generalizing a well known identity of Kummer, \cite{kummer}, to the hypergeometric functions of two or more variable like those of Appell and Lauricella denoted ${\rm F}_D^{(n)}.$ 

In the second part, using some reduction schemes of hyperelliptic integrals due to Goursat \cite{goursat}, Hermite \cite{her, herovres} we evaluate Appell and Lauricella's ${\rm F}_D^{(n)}$ hypergeometric functions and
their analytic continuations at some particular locations. 
Finally, by exploiting reductions of hyperelliptic integrals to elliptic due to Belokolos et al. \cite{Belokolos1986}, Eilbeck and Enol'skii \cite{eilbeck1994elliptic}, Enol'skii and Kostov \cite{enol1994geometry} and by Maier \cite{maier2008}, we obtain further links from multivariate hypergeometric
functions, to complete elliptic integrals and to  $\pi$. We thus
reach a conceptual settlement of the piece of research started by us in \cite{jnt1} and \cite{jnt2}. 

\vspace{0.5cm}


\end{abstract}

\section{Introduction and aim of the paper}
The hypergeometric functions' theoretical evaluation has become of great importance in order to provide a set of exact values for testing the numerical algorithms of their computation. But the greatest importance comes from such an aim accomplished through mathematical technicalities of remarkable power and intrinsic interest.

It was by hypergeometric series that two of the most remarkable chapters of classical function theory have been originated at the end of the XIX century, namely that on linear differential equations (I. L. Fuchs), and automorphic functions (H. Poincar\'e, F. Klein).
Besides, F. Klein showed that from the hypergeometric  equation, solved by the hypergeometric  function, really comes out the majority of the linear differential equations of mathematical physics, if not all.

The uncountable applications of such functions to mathematical physics have enabled numerical analysts in order to improve the quality, reliability and celerity of their effective computations.
 Coming to more recent developments, we can cite for instance three  papers by Joyce and Zucker, \cite{jz1, jz2, jz3}, where the theory of singular elliptic moduli is employed for a computational purpose.

In order to provide some contribution to the computability of such functions, and
as a reference for some values of them inside and outside their disks of convergence, we will proceed here according to the aims of \cite{jnt1} and \cite{jnt2}. 

In the first part of our paper, we obtain some new results in the spirit of theorem 3.1 of \cite{jnt2} where we proved that, if $2b>a>0$
\begin{equation}
\label{lunga}
_{2}\mathrm{F}_{1}\left( \left. 
\begin{array}{c}
2b-a;b \\[2mm]
2b
\end{array}
\right| 2\right) =(-i)^{2b-a}\,\sqrt\pi\,\frac{\Gamma\left(b+\dfrac12\right)}{\Gamma\left(\dfrac{a+1}{2}\right)\,\Gamma\left(\dfrac{2b-a+1}{2}\right)}.
\end{equation}
The above was obtained by our \textit{double evaluation method} (namely in terms of Gamma function and using the integral representation theorem, see \eqref{irt2f1} below) applied to the improper integral
\[
\int_0^\infty\frac{t^{a-1}}{(1+t^2)^b}{\rm d}t.
\]
Equation \eqref{lunga} can also be got by a quadratic transformation formula for the Gauss function $_{2}\mathrm{F}_{1},$ see for instance  \cite{stegun} entry 15.3.15 p. 560:
\[
_{2}\mathrm{F}_{1}\left( \left. 
\begin{array}{c}
a;b \\[2mm]
2b
\end{array}
\right| x\right)=(1-x)^{-a/2}\,_{2}\mathrm{F}_{1}\left( \left. 
\begin{array}{c}
\frac{a}{2};b-\frac{a}{2} \\[2mm]
b+\frac12
\end{array}
\right| \frac{x^2}{4(x-1)}\right)
\]
taken at $x=2,$ for which $x^2/(4(x-1))$ equates 1, when the right
hand hypergeometric function is evaluated as a gamma quotient. Nevertheless by the \textit{double evaluation method} a sequence of identities is obtained, see theorem 3.2, \cite{jnt2} formulae (3.10)--(3.13), whose, for instance, the first is:
\begin{equation}
_{2}\mathrm{F}_{1}\left( \left. 
\begin{array}{c}
\frac{1}{2};\frac{3}{4} \\[2mm]
\frac{3}{2}
\end{array}
\right| 2\right) =\frac{1}{2}\left( 1-i\right) \boldsymbol{K}\left( \frac{1}{
\sqrt{2}}\right).   \label{enu5:1}
\end{equation}
The above example explains the spirit of the present paper where, similarly, we obtain new formulae involving multivariate hypergeometric functions (Appell ${\rm F}_1$ and Lauricella ${\rm F}_D^{(n)}$) using two sequences of definite integrals
\begin{equation*}
A_n(a,b)=\int_0^1\frac{t^{a-1}}{(1-t^n)^b}{\rm d}t,\quad B_n(a,b)=\int_0^\infty\frac{t^{a-1}}{(1+t^n)^b}{\rm d}t. 
\end{equation*}
In this way we succeed in generalizing the Kummer identity, see \cite{kummer}, \cite{bailey} pp. 9-10, or Corollary~3.1.2  p. 126 in \cite{specfaar}:
\begin{equation} \label{ikummer}
_{2}\mathrm{F}_{1}\left( \left. 
\begin{array}{c}
a;b \\[2mm]
1+a-b
\end{array}
\right| -1\right)=\frac{\Gamma(1+a-b)\;\Gamma(1+\frac{a}{2})}
{\Gamma(1+a)\;\Gamma(1+\frac{a}{2}-b)}%=\frac{\sqrt{\pi}}{2^a}\,\frac{\Gamma(1+a-b)}{\Gamma\left(\frac{1+a}{2}\right)\,\Gamma\left(1+\frac{a}{2}-b\right)} %, \qquad (a-b\not\in\ZZ_{<0}).
\end{equation}
concerning the Gauss function, to such multivariate
hypergeometric functions as those of Appell and Lauricella.

In the second part of this paper, using some reduction schemes of hyperelliptic integrals due to Goursat \cite{goursat} and Hermite \cite{herovres, her}, we evaluate Appell and Lauricella ${\rm F}_D^{(n)}$ functions in some particular occurrences and in their analytic continuation, as in \cite{jnt1} and \cite{jnt2} where we obtained several identities connecting $\pi$, elliptic integrals and hypergeometric functions through reductions of hyperelliptic
integrals due to Jacobi \cite{cre} and to Hermite \cite{her}.

Finally, we take into account reductions of hyperelliptic integrals of increasing genus.
The relevant transformations have been introduced in past years by Belokolos et al. \cite{Belokolos1986}, Eilbeck and Enol'skii \cite{eilbeck1994elliptic}, Enol'skii and Kostov \cite{enol1994geometry} and produced in a systematic and algorithmic way by Maier \cite{maier2008}. 
Besides remarkable evaluations of functions of Appell, see for instance equation \eqref{hermyF1}, and  Lauricella, theorem \ref{genus4}, we go on with our path, see \cite{jnt1} and \cite{jnt2}, of using the reductions for gaining a double evaluation of the same definite integral. The hyperelliptic integrals are computed through the integral representation theorems, \eqref{irto}, while the identities are obtained by comparing the above evaluation to the relevant elliptic integral gained by the reduction.  Our integrals are necessary definite: therefore transformations of higher degree and/or integrals of higher genus would lead to algebraic equations of progressively higher degree, lacking an explicit solution in terms of radicals, and thus to an intractable mass of computations.

\section{Notations}
In this paper whenever $_{2}{\rm F}_1$ denotes the well-known
Gauss hypergeometric series: 
\begin{equation} \label{iseries}
_{2}\mathrm{F}_{1}\left( \left. 
\begin{array}{c}
a;b \\[2mm]
c
\end{array}
\right| x\right) =\sum_{m=0}^{\infty }\frac{\left( a\right) _{m}\left(
b\right) _{m}}{\left( c\right) _{m}}\frac{x^{m}}{m!},
\end{equation}
it also denotes its analytic continuation on
$\mathbb{C}\setminus [1,\infty)$ via its integral representation theorem
$
\re{a}>0,\,\re{c-a}>0,\,|x|<1$: 
\begin{equation}\label{irt2f1}
_{2}\mathrm{F}_{1}\left( \left. 
\begin{array}{c}
a;b \\[2mm]
c
\end{array}
\right| x\right) =\frac{\Gamma (c)}{\Gamma (c-a)\Gamma (a)}\int_{0}^{1}%
\frac{u^{a-1}(1-u)^{c-a-1}}{(1-x\,u)^{b}}\,\mathrm{d}u\,.
\end{equation}
In \eqref{iseries} $(a)_{m}$ is the usual Pochhammer symbol: 
\[
(a)_{m}=\frac{\Gamma (a+m)}{\Gamma (a)}=a(a+1)\cdots (a+m-1).
\]
We also consider the following multivariate extensions of $_{2}{\rm F}_1$:  the Appell ${\rm F}_1$ two-variable hypergeometric series, see \cite{ap}, defined for $%
|x_{1}|<1,\,|x_{2}|<1$: 
\[
\mathrm{F}_{1}\left( \left. 
\begin{array}{c}
a;b_{1},b_{2} \\[2mm]
c
\end{array}
\right| x_{1},x_{2}\right) =\sum_{m_{1}=0}^{\infty }\sum_{m_{2}=0}^{\infty }%
\frac{(a)_{m_{1}+m_{2}}(b_{1})_{m_{1}}(b_{2})_{m_{2}}}{(c)_{m_{1}+m_{2}}}%
\frac{x_{1}^{m_{1}}}{m_{1}!}\frac{x_{2}^{m_{2}}}{m_{2}!}.
\]
The analytic continuation of  Appell's function on
$\mathbb{C}\setminus [1,\infty)\times\mathbb{C}\setminus [1,\infty)$ comes from its integral representation theorem: if $\re{a}>0,\,\re{c-a}>0$: 
\begin{equation}%\label{irtoappel}
\mathrm{F}_{1}\left( \left. 
\begin{array}{c}
a;b_{1},b_{2} \\[2mm]
c
\end{array}
\right| x_{1},x_{2}\right) =\frac{\Gamma (c)}{\Gamma (a)\Gamma (c-a)}%
\int_{0}^{1}\frac{u^{a-1}\left( 1-u\right) ^{c-a-1}}{\left(
1-x_{1}\,u\right) ^{b_{1}}\left( 1-x_{2}\,u\right) ^{b_{2}}}\,\mathrm{d}u.
\label{F1}
\end{equation}
Its $n$-variable extension leads to Lauricella hypergeometric series, introduced in \cite{Lau}:
\[
\mathrm{F}_{D}^{(n)}\left( \left. 
\begin{array}{c}
a;b_{1},\ldots ,b_{n} \\[2mm]
c
\end{array}
\right| x_{1},\ldots ,x_{n}\right) =\sum_{m_{1}=0}^{\infty }\cdots
\sum_{m_{n}=0}^{\infty }\frac{(a)_{m_{1}+\cdots +m_{n}}(b_{1})_{m_{1}}\cdots
(b_{n})_{m_{n}}}{(c)_{m_{1}+\cdots +m_{n}}m_{1}!\cdots m_{n}!}%
\,x_{1}^{m_{1}}\cdots x_{n}^{m_{m}}
\]
whose integral representation theorem for $\re{a}>0,\,\re{c-a}>0$ allows its analytic continuation on the $n$-fold cartesian product of $\mathbb{C}\setminus [1,\infty)^n$
\begin{equation}\label{irto}
\mathrm{F}_{D}^{(n)}\left( \left. 
\begin{array}{c}
a;b_{1},\ldots ,b_{n} \\[2mm]
c
\end{array}
\right| x_{1},\ldots ,x_{n}\right) =\frac{\Gamma (c)}{\Gamma (a)\,\Gamma
(c-a)}\,\int_{0}^{1}\frac{u^{a-1}(1-u)^{c-a-1}}{(1-x_{1}u)^{b_{1}}\cdots
(1-x_{n}u)^{b_{n}}}\,\mathrm{d}u. 
\end{equation}
We will also use the order reduction formula for Lauricella functions
\begin{equation} \label{id:3}
\mathrm{F}_{D}^{(n)}\left( \left. 
\begin{array}{c}
a;b_{1},\ldots ,b_{n} \\[2mm]
b_{1}+\cdots +b_{n}
\end{array}
\right| x_{1},\ldots ,x_{n}\right) =\frac{1}{(1-x_{n})^{a}}\,\mathrm{F}%
_{D}^{(n-1)}\left( \left. 
\begin{array}{c}
a,b_{1},\ldots ,b_{n-1} \\ 
b_{1}+\cdots +b_{n}
\end{array}
\right| \frac{x_{1}-x_{n}}{1-x_{n}},\ldots ,\frac{x_{n-1}-x_{n}}{1-x_{n}}%
\right)  
\end{equation}
whose proof was given in \cite{jnt2}, lemma 1.1 therein.

\section{Special  hypergeometric values from Eulerian integrals}

\subsection{Values at the boundary of unit disk}
For the purpose, let us introduce a generalization of the Kummer identity to both Appell and Lauricella functions. Our starting point is an elementary proof of \eqref{ikummer} we discovered, at least we presume, through a two-fold evaluation of the definite integral:
\begin{equation}\label{integrale}
A_2(a,b)=\int_0^1\frac{t^{a-1}}{(1-t^2)^b}{\rm d}t
\end{equation}
where, for the sake of convergence of \eqref{integrale}, we assume $\re{a}>0$ and $\re{b}<1.$ 
First of all, we change variable in \eqref{integrale}, putting $t^2=u;$ it follows that
\[
A_2(a,b)=\frac12\frac{\Gamma\left(\frac{a}{2}\right)\Gamma(1-b)}{\Gamma\left(1+\frac{a}{2}-b\right)}.
\]
On the other hand, writing \eqref{integrale} as
\begin{equation}\label{integrale2}
A_2(a,b)=\int_0^1\frac{t^{a-1}(1-t)^{-b}}{(1+t)^b}{\rm d}t,
\end{equation}
from the \eqref{irt2f1} integral representation theorem, we get:
\begin{equation}\label{A2}
A_2(a,b)=\frac{\Gamma(a)\Gamma(1-b)}{\Gamma(1-a-b)}\,_{2}\mathrm{F}_{1}\left( \left. 
\begin{array}{c}
a;b \\[2mm]
1+a-b
\end{array}
\right| -1\right).
\end{equation}
By comparing the expressions for $A_2(a,b)$ and recalling that
\[
\frac{\Gamma\left(\frac{a}{2}\right)}{2\Gamma(a)}=\frac{\Gamma\left(1+\frac{a}{2}\right)}{\Gamma(1+a)}
\]
we infer \eqref{ikummer} which was originally found by Kummer in \cite{kummer}. 
Notice that our method of the double evaluation allows avoiding a quadratic transformation in order to prove the Kummer identity, see the proof of \eqref{ikummer} given in \cite{specfaar} corollary 3.1.2 p. 136.

To obtain for the Appell ${\rm F}_1$ function something analogous to formula \eqref{ikummer}, we consider the integral, where we change the exponent 2 in \eqref{integrale} to 3:
\begin{equation}\label{integrale3}
A_3(a,b)=\int_0^1\frac{t^{a-1}}{(1-t^3)^b}{\rm d}t.
\end{equation}
Putting $t^3=u$ we evaluate \eqref{integrale3} through the Eulerian integral of the first kind as
\[
A_3(a,b)=\frac13\frac{\Gamma\left(\frac{a}{3}\right)\Gamma(1-b)}{\Gamma\left(1+\frac{a}{3}-b\right)}.
\]
On the other hand we can factorize $1-t^3$ obtaining
\[
%\begin{split}
A_3(a,b)=\int_0^1\frac{t^{a-1}(1-t)^{-b}}{(1-e^{\frac{2}{3}\pi i}t)^b(1-e^{-\frac{2}{3}\pi i}t)^b}{\rm d}t=\frac{\Gamma(a)\Gamma(1-b)}{\Gamma(1+a-b)}\,{\rm F}_1\left( \left. 
\begin{array}{c}
a;b,b \\[2mm]
1 + a - b
\end{array}
\right| e^{\frac{2}{3}\pi i},e^{-\frac{2}{3}\pi i}\right).
%\end{split}
\]
Using the elementary identity, $n\in\mathbb{N}:$
\[
\frac{\Gamma\left(\frac{a}{n}\right)}{n\Gamma(a)}=\frac{\Gamma\left(1+\frac{a}{n}\right)}{\Gamma(1+a)},
\] 
taking $n=3$ and comparing the expression for $A_3(a,b)$ we get
\begin{equation}\label{effe1}
{\rm F}_1\left( \left. 
\begin{array}{c}
a;b,b \\[2mm]
1 + a - b
\end{array}
\right| e^{\frac{2}{3}\pi i},e^{-\frac{2}{3}\pi i}\right)=\frac{\Gamma(1+a-b)\;\Gamma(1+\frac{a}{3})}
{\Gamma(1+a)\;\Gamma(1+\frac{a}{3}-b)}.
\end{equation}
Equation \eqref{effe1} requires comments. It is a two-variable generalization of the Kummer identity \eqref{ikummer}: can it also be obtained through  some variable transformation for instance of quadratic or cubic kind? Such transformations after Appell and Kamp{\'e} de F{\'e}riet  \cite{ap}, were pursued by Erd\'elyi in \cite{mr00261160}, a paper concerning only the function ${\rm F}_2$; and, more recently by Carlson \cite{carlson1976quadratic} and Matsumoto \cite{matsumoto2010transformation} (quadratic transformations) and by Matsumoto and Ohara \cite{matsumoto2009some}, who analyze a cubic transformation. More particularly the transformations studied in  \cite{matsumoto2010transformation} and in \cite{matsumoto2009some},  depending upon only one free parameter, have the parameters meeting the ratio of formula \eqref{effe1}, namely:
\[
\begin{split}
&(z_1z_2)^{(1-c)/2}\left(\frac{z_1+z_2}{2}\right)^c\,\mathrm{F}_{1}\left( \left.
\begin{array}{c}
\frac{3+c}{4},\frac{1+c}{4};\frac{1+c}{6} \\[2mm] 
\frac{3+3c}{4}
\end{array}
\right| 1-z_1^2,1-z_2^2\right)\\
&=\mathrm{F}_{1}\left(\left. 
\begin{array}{c}
c;\frac{1+c}{4};\frac{1+c}{4} \\[2mm] 
\frac{c+5}{6}
\end{array}
\right| 1-\frac{z_1(1+z_2)}{z_1+z_2},1-\frac{z_2(1+z_1)}{z_1+z_2}\right)
\end{split}
\]
and
\[
\left(\frac{1+z_1+z_2}{3}\right)^c\,\mathrm{F}_{1}\left( \left.
\begin{array}{c}
\frac{c}{3},\frac{c+1}{6};\frac{c+1}{6} \\[2mm] 
\frac{c+1}{2}
\end{array}
\right| 1-z_1^3,1-z_2	^3\right)=\mathrm{F}_{1}\left( \left.
\begin{array}{c}
\frac{c}{3};\frac{c+1}{6};\frac{c+1}{6} \\[2mm] 
\frac{c+5}{6}
\end{array}
\right| z'_1,z'_2\right)
\]
where $\omega=\dfrac{-1+i\sqrt3}{2}$ and $(z_1,z_2)$ is linked to $(z'_1,z'_2)$ by:
\[
(z_1,z_2)\mapsto(z'_1,z'_2)=\left(\left(\frac{1+\omega z_1+\omega^2z_2}{1+z_1+z_2}\right)^3,\left(\frac{1+\omega^2z_1+\omega z_2}{1+z_1+z_2}\right)^3\right).
\]
However, we do not think that our identity  \eqref{effe1} can be obtained from these. The same is true if we consider the quadratic transformation given by equation (4.1) in Carlson paper \cite{carlson1976quadratic}, i.e.
\[
\mathrm{F}_{1}\left( \left.
\begin{array}{c}
2a;b,b \\[2mm] 
a+b+\frac12
\end{array}
\right| \sin^2\theta,\sin^2\varphi\right)
=\mathrm{F}_{1}\left( \left.
\begin{array}{c}
b;a,a \\[2mm] 
a+b+\frac12
\end{array}
\right| \sin^2(\theta+\varphi),\sin^2(\theta-\varphi)\right)
\]
Of course the settlement of a change of variable for function ${\rm F}_1$ capable of yielding \eqref{effe1}, is a possible subject of a further research. 

We terminate our discussion about \eqref{effe1} noting that, by applying the first degree \lq\lq Pfaff-like'' transformation for the Appell ${\rm F}_1$ function:
\[
\mathrm{F}_{1}\left( \left.
\begin{array}{c}
a;b_1,b_2 \\[2mm] 
c
\end{array}
\right| x_1,x_2\right)=(1-x_1)^{-b_{1}}(1-x_2)^{-b_{2}}\mathrm{F}_{1}\left( \left.
\begin{array}{c}
c-a;b_1,b_2 \\[2mm] 
c
\end{array}
\right| \frac{x_1}{x_1-1},\frac{x_2}{x_2-1}\right)
\]
where we have the limitation ${\rm Re\, }c>{\rm Re\,}a>0,\,|x_k|<1,\,|x_k/(x_k-1)|<1$, we obtain, in elementary way, more identities since using the Pfaff transformation we can rewrite \eqref{effe1} as
\begin{equation}\tag{\ref{effe1}b}
{\rm F}_1\left( \left. 
\begin{array}{c}
1-b;b,b \\[2mm]
1 + a - b
\end{array}
\right|\frac12\left(1-\frac{i}{\sqrt3}\right),\frac12\left(1+\frac{i}{\sqrt3}\right)\right)=3^{b}\,\frac{\Gamma(1+a-b)\;\Gamma(1+\frac{a}{3})}
{\Gamma(1+a)\;\Gamma(1+\frac{a}{3}-b)}.
\end{equation}

Going on, from the integral
\begin{equation}\label{integrale4}
\int_0^1\frac{t^{a-1}}{(1-t^4)^b}{\rm d}t=\frac{\Gamma \left(\frac{a}{4}\right) \Gamma (1-b)}{4 \Gamma
   \left(1+\frac{a}{4}-b\right)}
\end{equation}
factorizing $1-t^4$ we get the Kummer-like formula for the Lauricella ${\rm F}_D^{(3)}$ and the companion formula obtained after a \lq\lq Pfaff-like'' degree-1 transformation of the Lauricella ${\rm F}_D^{(3)}$:
\begin{equation}\label{fd3}
\mathrm{F}_{D}^{(3)}\left( \left. 
\begin{array}{c}
a;b,b,b \\[2mm]
1+a-b
\end{array}
\right| -1,i,-i\right)=\frac{\Gamma (a-b+1)\Gamma \left(1+\frac{a}{4}\right) }{\Gamma (1+a)
   \Gamma \left(1+\frac{a}{4}-b\right)}
\end{equation}
\begin{equation}\tag{\ref{fd3}b}
\mathrm{F}_{D}^{(3)}\left( \left. 
\begin{array}{c}
1-b;b,b,b \\[2mm]
1+a-b
\end{array}
\right| \frac12,\frac12(1-i),\frac12(1+i)\right)=4^b\,\frac{\Gamma (a-b+1)\Gamma \left(1+\frac{a}{4}\right) }{\Gamma (1+a)
   \Gamma \left(1+\frac{a}{4}-b\right)}
\end{equation}
We see that the general relation for $n\in\mathbb{N},\,n\geq2$ is
\begin{equation}\label{rela0}
A_n(a,b)=\frac{\Gamma(\frac{a}{n})\Gamma(1-b)}{n\Gamma(1+\frac{a}{n}+b)}.
\end{equation}
Thus by \eqref{rela0} we can state our result in general. In order to simplify the notation whenever in a Lauricella function the $n$ parameters of type $b$ are all equal, $b_{1}=\cdots=b_{n}=b$, we put
\[
\mathrm{F}_{D}^{(n)}\left( \left. 
\begin{array}{c}
a;b_{1},\ldots ,b_{n} \\[2mm]
c
\end{array}
\right| x_{1},\ldots ,x_{n}\right) =\mathrm{F}_{D}^{(n)}\left( \left. 
\begin{array}{c}
a;b \\[2mm]
c
\end{array}
\right| x_{1},\ldots ,x_{n}\right)
\]
the number of repetitions of $b$ is denoted by the apex of the Lauricella function. The theorem generalizing \eqref {fd3} is then the following:

\begin{theorem}
Let $n$ be a positive integer $\geq2$. For $k=1,\ldots,n-1$ let $\omega_{k,n}=e^{i\frac{2k\pi}{n}}$ root of unity so that $\omega_{k,n}^n=1$ and $\omega_{k,n}\neq1$ for $k=1,\dots,n-1.$ Let $0<\re{b}<1$ and $\re{a}>0$; then
\begin{equation}\label{fdn}
\mathrm{F}_{D}^{(n-1)}\left( \left. 
\begin{array}{c}
a,b \\[2mm]
1+a-b
\end{array}
\right|\omega_{1,n},\ldots ,\omega_{n-1,n}\right) =\frac{\Gamma (a-b+1)\Gamma \left(1+\frac{a}{n}\right) }{\Gamma (1+a)
   \Gamma \left(1+\frac{a}{n}-b\right)}
\end{equation}
\begin{equation}\tag{\ref{fdn}b}
\mathrm{F}_{D}^{(n-1)}\left( \left. 
\begin{array}{c}
1-b,b \\[2mm]
1+a-b
\end{array}
\right|\frac{\omega_{1,n}}{\omega_{1,n}-1},\ldots ,\frac{\omega_{n-1,n}}{\omega_{n-1,n}-1}\right) =n^{b}\frac{\Gamma (a-b+1)\Gamma \left(1+\frac{a}{n}\right) }{\Gamma (1+a)
   \Gamma \left(1+\frac{a}{n}-b\right)}
\end{equation}
\end{theorem}
\begin{proof}
From the integral representation theorem, \eqref{irto} we can establish relation \eqref{rela1} below. It provides integrals $A_n(a,b)$ in terms of $(n-1)$ multivariate Lauricella function $\mathrm{F}_{D}^{(n-1)}$:
\begin{equation}\label{rela1}
A_n(a,b)=\frac{\Gamma(a)\Gamma(1-b)}{\Gamma(1+a-b)}\mathrm{F}_{D}^{(n-1)}\left( \left. 
\begin{array}{c}
a,b \\[2mm]
1+a-b
\end{array}
\right|\omega_{1,n},\ldots ,\omega_{n-1,n}\right)
\end{equation}
Thesis \eqref{fdn} follows by comparing \eqref{rela1} with \eqref{rela0} and (\ref{fdn}b) follows from the Pfaff first degree transformation of the Lauricella function $\mathrm{F}_{D}^{(n-1)},$ recalling that for $n\in\mathbb{N}$
\[
\prod_{k=1}^{n-1}\left(1-\omega_{k,n}\right)=\prod_{k=1}^{n-1}\left(1-\cos\frac{2k\pi}{n}+i\,\sin\frac{2k\pi}{n}\right)=n.
\]
\end{proof}
%\subsubsection*{Remark}
\begin{remark}
In the case $n=2$ when $a=1$ and $b=1/2$ we have the elementary value $A_2(1,1/2)=\pi/2$ which leads to
\[
_2{\rm F}_1\left( \left. 
\begin{array}{c}
1,1/2 \\[2mm]
3/2
\end{array}
\right|-1\right)=\frac{\pi}{4}
\]
according to equation (14) section 2.8 p. 102 of \cite{bateman}:
\[
\frac{\arctan z}{z}=\, _2{\rm F}_1\left( \left. 
\begin{array}{c}
1,1/2 \\[2mm]
3/2
\end{array}
\right|-z^2\right).
\]
\end{remark}
\subsection{Values outside the unit disk}
Consider the sequence of integrals
\begin{equation}\label{grad}
B_n(a,b)=\int_0^\infty\frac{t^{a-1}}{(1+t^n)^b}{\rm d}t
\end{equation}
where we assume $a>0,\,b>0,\,na>b$ to ensure convergence. The integral  \eqref{grad} can be computed by means of the Mellin transform in \cite{bat} p. 310 equation 21, see also \cite{gra} entry 3.194-3 p. 313, so that
\[
B_n(a,b)=\frac{\Gamma(\frac{a}{n})\Gamma(\frac{nb-a}{n})}{n\Gamma(b)}.
\]
On the other hand, we can also evaluate $B_n(a,b)$ using the integral representation theorem, \eqref{irto}. In fact the change of variable $t=(1-u)/u$ leads to
\begin{equation}\label{Jn}
B_n(a,b)=\int_0^1\frac{u^{nb-a-1}(1-u)^{a-1}}{\left(u^n+(1-u)^n\right)^b}{\rm d}u.
\end{equation}
Notice that for $n$ even the polynomial $u^n+(1-u)^n$ has degree $n$, while for $n$ odd has degree $n-1$. In any case the roots are distinct and the factorization holds:
\begin{equation*}%\label{even}
u^n+(1-u)^n=\prod_k\left(1-\frac{1}{\alpha_k}u\right)
\end{equation*}
where $\alpha_k$ is a root of $u^n+(1-u)^n=0$. These roots are evaluated, through the De Moivre formula, as follows. If $n$ is even, say $n=2m$, we have
\begin{equation}\label{evenroot}
\alpha_k=\frac12-i\frac{\sin\frac{(2k-1)\pi}{2m}}{2\left(1+\cos\frac{(2k-1)\pi}{2m}\right)},\quad k=1,\dots,2m.
\end{equation}
If $n$ is odd, say $n=2m-1$
\begin{equation}\label{oddroot}
\alpha_k=\frac12-i\frac{\sin\frac{(2k-1)\pi}{2m-1}}{2\left(1+\cos\frac{(2k-1)\pi}{2m-1}\right)},\quad k=1,\dots,2m-1,\,k\neq m.
\end{equation}

Now let us state our second result. 

\begin{theorem}
Assume that $a>0,\,b>0,\,nb>a.$ If $n=2m$ is even, then:
\begin{equation}\label{even}
\mathrm{F}_{D}^{(2m)}\left( \left. 
\begin{array}{c}
2mb-a;b \\[2mm]
2mb
\end{array}
\right| x_1,\dots,x_{2m}\right)=\frac{1}{2m}\,\frac{\Gamma(\frac{a}{2m})\Gamma(2mb)\Gamma(\frac{2mb-a}{2m})}{\Gamma(a)\Gamma(b)\Gamma(2mb-a)}
\end{equation}
where, in \eqref{even} for $k=1,\dots,2m$:
\begin{equation}\label{reci}
x_k=1+\cos\frac{(2k-1)\pi}{2m}+i\sin\frac{(2k-1)\pi}{2m}.
\end{equation}
If $n=2m-1$ is odd then
\begin{equation}\label{odd}
\mathrm{F}_{D}^{(2m-2)}\left( \left. 
\begin{array}{c}
(2m-1)b-a;b \\[2mm]
(2m-1)b
\end{array}
\right| y_1,\dots,y_{2m-1}\right)=\frac{1}{2m-1}\,\frac{\Gamma(\frac{a}{2m-1})\Gamma((2m-1)b)\Gamma(\frac{(2m-1)b-a}{2m-1})}{\Gamma(a)\Gamma(b)\Gamma((2m-1)b-a)}
\end{equation}
where in \eqref{odd} for $k=1,\dots,2m-1,\,k\neq m$ 
\[
y_k=1+\cos\frac{(2k-1)\pi}{2m-1}+i\sin\frac{(2k-1)\pi}{2m-1}.
\]
\end{theorem}
%\subsubsection*{Remark}
\begin{remark}
In the even case $n=2m$, it is possible, using Lemma 1.1 of \cite{jnt2}, to reduce the order of the Lauricella function appearing in \eqref{even}, obtaining:
\[
\left(-e^{\frac{i \pi }{2 m}}\right)^{2 b m-a}\mathrm{F}_{D}^{(2m-1)}\left( \left. 
\begin{array}{c}
2mb-a;b \\[2mm]
2mb
\end{array}
\right| z_1,\dots,z_{2m-1}\right)=\frac{1}{2m}\,\frac{\Gamma(\frac{a}{2m})\Gamma(2mb)\Gamma(\frac{2mb-a}{2m})}{\Gamma(a)\Gamma(b)\Gamma(2mb-a)}
\]
with arguments given for $k=1,\dots,2m-1$ by
\[
z_k=1-\cos\frac{\pi k}{m}-i \sin\frac{\pi  k}{m}.
\]
Thus for $m=1$ we obtain, as a particular case, theorem 3.1 of \cite{jnt2}.
\end{remark}
\begin{proof}
Using the integral representation theorem, formula \eqref{irto} in equation \eqref{Jn} we get, if $n=2m$ is even
\begin{equation}\label{repeven}
B_n(a,b)=\frac{\Gamma(a)\Gamma(nb-a)}{\Gamma(nb)}\mathrm{F}_{D}^{(n)}\left( \left. 
\begin{array}{c}
nb-a;b \\[2mm]
nb
\end{array}
\right|x_1,\dots,x_n\right)
\end{equation}
where  we denote by $x_1,\dots,x_n$ the reciprocal of the $n$ roots of equation $u^n+(1-u)^n=0$ given by \eqref{reci}. While if $n=2m-1$ is odd we have
\begin{equation}\label{repodd}
B_n(a,b)=\frac{\Gamma(a)\Gamma(nb-a)}{\Gamma(nb)}\mathrm{F}_{D}^{(n-1)}\left( \left. 
\begin{array}{c}
nb-a;b \\[2mm]
nb
\end{array}
\right| x_1,\dots,x_n\right)
\end{equation}
but here the root $x_{\frac{n+1}{2}}=x_m$ is skipped. Theses \eqref{even} and \eqref{odd} follow by comparing \eqref{repeven} and \eqref{repodd} with  \eqref{grad}.
\end{proof}
\section{Some special cases}
Whenever in some special cases the Eulerian integrals $A_n(a,b)$ and $B_n(a,b)$ can be evaluated through complete elliptic integrals of the first and second kinds, we are able to establish, from our formulae \eqref{fdn} and \eqref{odd}, new relationships of hypergeometric functions' special values to certain complete elliptic integrals. These relationships arise for singular moduli of the elliptic integrals, according to the relevant theory studied by several authors, see \cite{SC, WW, BB, BZ}. We will use some integrals taken from \cite{gra} and some less known ones exploited in the past by eminent mathematicians of the XIX century, like Legendre \cite{tr-1} and \cite{tr-2}, Richelot \cite{R}, Serret \cite{S}. Such integrals, as far as we know, have been forgotten by recent literature.

Starting with the following elliptic integrals, notice that in any occurrence we also provide the integral under investigation through our families $A_n(a,b)$ or $B_n(a,b)$
\begin{align}\label{k1:2}
B_4(1,\tfrac12)&=\int_0^\infty\frac{{\rm d}x}{\sqrt{1+x^4}}=\boldsymbol{K}\left(\frac{1}{\sqrt2}\right)\\
B_3(\tfrac12,\tfrac12)&=\int_0^\infty\frac{{\rm d}x}{\sqrt{x(1+x^3)}}=\boldsymbol{K}\left(\frac{\sqrt6-\sqrt2}{4}\right)\label{k:r6r2}
\end{align}
By \eqref{repeven} and \eqref{repodd} we get
\begin{align}\label{k1:2rep}
&\mathrm{F}_{D}^{(4)}\left( \left. 
\begin{array}{c}
1;\frac12 \\[2mm]
2
\end{array}
\right| 1+\frac{1+i}{\sqrt{2}},1-\frac{1-i}{\sqrt{2}},1-\frac{1+i}{\sqrt{2}},1+\frac{1-i}{\sqrt{2}}\right)=\boldsymbol{K}\left(\frac{1}{\sqrt2}\right)\\
&\mathrm{F}_{1}\left( \left. 
\begin{array}{c}
1;\frac12,\frac12 \\[2mm]
3/2
\end{array}
\right| \frac32+i\frac{\sqrt3}{2},\frac32-i\frac{\sqrt3}{2}\right)=\frac{2}{\sqrt[4]{27}}\boldsymbol{K}\left(\frac{\sqrt6-\sqrt2}{4}\right)\label{k:r6r2rep}
\end{align}
By Lemma 1.1 of \cite{jnt2} in \eqref{k1:2rep} we can also establish values of ${\rm F}_{D}^{(3)}$ when one of the variables is in the positive real axis in analogy with our results for the Gauss $_2{\rm F}_1$ exposed in \cite{jnt2} theorem 3.2 equations (27) to (30) therein:
\[
\mathrm{F}_{D}^{(3)}\left( \left. 
\begin{array}{c}
1;\frac12 \\[2mm]
2
\end{array}
\right| 1 - i, 2, 1 + i\right)=-\frac{1-i}{\sqrt{2}}\boldsymbol{K}\left(\frac{1}{\sqrt2}\right)
\]

Now we turn our attention to some integrals from classical repertories \cite{gra} and \cite{by}. We begin with integral 3.183.2 p. 313 of \cite{gra}:
\[
A_2(1,\tfrac14)=\int_0^1\frac{{\rm d}x}{\sqrt[4]{1-x^2}}=\sqrt2\left[2\boldsymbol{E}\left(\frac{1}{\sqrt2}\right)-\boldsymbol{K}\left(\frac{1}{\sqrt2}\right)\right]
\]
so that, comparing with \eqref{A2}, we get the formula:
\begin{equation}\label{3.183.2}
_2\mathrm{F}_1\left( \left. 
\begin{array}{c}
1;\frac14 \\[2mm]
\frac74
\end{array}
\right|-1\right)=\frac34\sqrt2\left[2\boldsymbol{E}\left(\frac{1}{\sqrt2}\right)-\boldsymbol{K}\left(\frac{1}{\sqrt2}\right)\right]
\end{equation}
 where $ \boldsymbol{E}$ and  $ \boldsymbol{K}$ are the complete elliptic integrals of second and first kind respectively.
 From entry 3.184.1 p. 314 of \cite{gra}
\[
A_2(3,\tfrac14)=\int_0^1\frac{x^2}{\sqrt[4]{1-x^2}}{\rm d}x=\frac{2\sqrt2}{5}\left[2\boldsymbol{E}\left(\frac{1}{\sqrt2}\right)-\boldsymbol{K}\left(\tfrac{1}{\sqrt2}\right)\right]
\]
we obtain
\begin{equation}\label{3.184.1}
_2\mathrm{F}_1\left( \left. 
\begin{array}{c}
3;\frac14 \\[2mm]
\frac{15}{4}
\end{array}
\right|-1\right)=\frac{231\sqrt2}{320}\left[2\boldsymbol{E}\left(\frac{1}{\sqrt2}\right)-\boldsymbol{K}\left(\frac{1}{\sqrt2}\right)\right]
\end{equation}
Using formula 3.185.2 p. 314 of \cite{gra}
\[
A_2(1,\tfrac34)=\int_0^1\frac{{\rm d}x}{\sqrt[4]{(1-x^2)^3}}=\sqrt2 \boldsymbol{K}\left(\frac{1}{\sqrt2}\right)
\]
we infer
\begin{equation}\label{3.185.2}
_2\mathrm{F}_1\left( \left. 
\begin{array}{c}
1;\frac34 \\[2mm]
\frac{5}{4}
\end{array}
\right|-1\right)=\frac{\sqrt2}{4}\boldsymbol{K}\left(\frac{1}{\sqrt2}\right)
\end{equation}
Let us end the examples taken from \cite{gra} with entry 3.185.4 p. 314
\[
A_2(3,\tfrac34)=\int_0^1\frac{x^2}{\sqrt[4]{(1-x^2)^3}}{\rm d}x=\frac{2\sqrt2}{3} \boldsymbol{K}\left(\frac{1}{\sqrt2}\right)
\]
which leads to
\begin{equation}\label{3.185.4}
_2\mathrm{F}_1\left( \left. 
\begin{array}{c}
3;\frac34 \\[2mm]
\frac{13}{4}
\end{array}
\right|-1\right)=\frac{15}{32 \sqrt{2}}\boldsymbol{K}\left(\frac{1}{\sqrt2}\right)
\end{equation}

We start from entries 576.00 p. 256 and 578.00 p. 258 of \cite{by}
\begin{align}
\int_0^1\frac{{\rm d}x}{\sqrt{1-x^6}}=\frac{1}{\sqrt[4]{3}}\boldsymbol{K}\left(\frac{\sqrt6-\sqrt2}{4}\right)\label{576.00}\\
\int_0^\infty\frac{{\rm d}x}{\sqrt{1+x^6}}=\frac{2}{\sqrt[4]{27}}\boldsymbol{K}\left(\frac{\sqrt6-\sqrt2}{4}\right)\label{578.00}
\end{align}
Comparing \eqref{576.00} with \eqref{rela1} and \eqref{578.00} with \eqref{repeven} we obtain
\begin{align}
\label{576.00b}&\mathrm{F}_D^{(5)}\left( \left. 
\begin{array}{c}
1;\frac12 \\[2mm]
\frac{3}{2}
\end{array}
\right|\,x^{(6)}_1,x^{(6)}_2,x^{(6)}_3,x^{(6)}_4,x^{(6)}_5\right)=\frac{1}{4\sqrt[4]{3}}\,\boldsymbol{K}\left(\frac{\sqrt6-\sqrt2}{4}\right)\\
\label{578.00b}   &\mathrm{F}_D^{(6)}\left( \left. 
\begin{array}{c}
2;\frac12 \\[2mm]
3
\end{array}
\right|\,y^{(6)}_1,y^{(6)}_2,y^{(6)}_3,y^{(6)}_4,y^{(6)}_5,y^{(6)}_6\right)=\frac{4}{\sqrt[4]{27}}\,\boldsymbol{K}\left(\frac{\sqrt6-\sqrt2}{4}\right)
\end{align} 
where
\[
x^{(6)}_1=\frac{1}{2}+\frac{i \sqrt{3}}{2},\quad x^{(6)}_2-\frac{1}{2}+\frac{i
   \sqrt{3}}{2},\quad x^{(6)}_3=-1,\quad
   x^{(6)}_4=-\frac{1}{2}-\frac{i \sqrt{3}}{2},\quad x^{(6)}_5=\frac{1}{2}-\frac{i
   \sqrt{3}}{2}
\]
and
\[\begin{split}
&y^{(6)}_1=1+\frac{\sqrt{3}}{2}+\frac{i}{2},\quad y^{(6)}_2=1+i,\quad y^{(6)}_3=1-\frac{\sqrt{3}}
   {2}+\frac{i}{2}\\
  & y^{(6)}_4=1-\frac{\sqrt{3}}{2}-\frac{i}{2},\quad y^{(6)}_5=1-i,\quad y^{(6)}_6=1+\frac{\sqrt{3}}{2}-\frac{i}{2}
\end{split}
\]
Putting \eqref{repodd} in \eqref{578.00b}, we evaluate $\mathrm{F}_D^{(5)}$ when one of its arguments is 2. Defining
\[
z^{(6)}_1=\frac{1}{2}-\frac{i \sqrt{3}}{2},\quad z^{(6)}_2=\frac{3}{2}-\frac{i \sqrt{3}}{2},\quad z^{(6)}_3=2,\quad z^{(6)}_4=\frac{3}{2}+\frac{i \sqrt{3}}{2},\quad z^{(6)}_5=\frac{1}{2}+\frac{i
   \sqrt{3}}{2}
\]
we have
\begin{equation}\label{serretprol}
\mathrm{F}_D^{(5)}\left( \left. 
\begin{array}{c}
2;\frac12 \\[2mm]
3
\end{array}
\right|\,z^{(6)}_1,\,z^{(6)}_2,\,z^{(6)}_3,\,z^{(6)}_4,\,z^{(6)}_5,\right)=\frac{4}{\sqrt[4]{27}}\left(-\frac{\sqrt{3}}{2}+\frac{i}{2}\right)\,\boldsymbol{K}\left(\frac{\sqrt6-\sqrt2}{4}\right)
\end{equation}

Now let us pass to some integrals studied by Legendre, Richelot and Serret. Legendre, \cite{tr-2} p. 383, proved that 
\begin{equation}\label{leg1}
A_8(1,\tfrac12)=\int_0^1\frac{{\rm d}x}{\sqrt{1-x^8}}=\frac{1}{\sqrt2}\boldsymbol{K}\left({\sqrt2}-1\right),
\end{equation}
thus, comparing \eqref{leg1} with \eqref{rela1}, we get:
\begin{equation}
\mathrm{F}_D^{(7)}\left( \left. 
\begin{array}{c}
1;\frac12 \\[2mm]
\frac{3}{2}
\end{array}
\right|\frac{1+i}{\sqrt{2}},i,-\frac{1-i}{\sqrt{2}},-1,-\frac{1+i}{\sqrt{2}},-i,\frac{1-i}{\sqrt{2}}\right)=\frac{1}{2\sqrt2}\,\boldsymbol{K}\left({\sqrt2}-1\right)
\end{equation} 
Richelot in \cite{R} evaluated
\begin{equation}\label{ric1}
A_8(3,\tfrac12)=\int_0^1\frac{x^2}{\sqrt{1-x^8}}{\rm d}x=\left(1-\frac{1}{\sqrt2}\right)\boldsymbol{K}\left({\sqrt2}-1\right)
\end{equation} 
thus comparing \eqref{ric1} with \eqref{rela1} we obtain
\begin{equation}
\mathrm{F}_D^{(7)}\left( \left. 
\begin{array}{c}
3;\frac12 \\[2mm]
\frac{7}{2}
\end{array}
\right|\frac{1+i}{\sqrt{2}},i,-\frac{1-i}{\sqrt{2}},-1,-\frac{1+i}{\sqrt{2}},-i,\frac{1-i}{\sqrt{2}}\right)=\frac{15}{16}\left(1-\frac{1}{2\sqrt2}\right)\,\boldsymbol{K}\left({\sqrt2}-1\right)
\end{equation}
To establish further identities, we recall two remarkable formulae due to Legendre, and unrecognized by mathematicians of our age. In his famous Trait\'e, \cite{tr-2} on p. 377, equations (z) therein, we find, for  $2a<n$
\begin{equation}\label{z1}
\int_0^1\frac{x^{a-1}}{\sqrt{1-x^n}}\,{\rm d}x=\cos\left(\frac{a}{n}\pi\right)\int_0^\infty\frac{z^{a-1}}{\sqrt{1+z^n}}\,{\rm d}z.
\end{equation}
The second, for $n/2<a<n$, is
\begin{equation}\label{z2}
\int_0^\infty\frac{z^{n-a-1}}{\sqrt{1+z^n}}\,{\rm d}z\cdot\int_0^1\frac{x^{a-1}}{\sqrt{1-x^n}}\,{\rm d}x=\frac{2\pi}{n(2a-n)\sin(\frac{\pi}{n}a)}.
\end{equation}
We will not refer here the Legendre's path for establishing \eqref{z1} and \eqref{z2}, a beautiful treatment, whose original source is now easily accessible. We remark that \eqref{z1} plays a role in evaluating the integral \eqref{leg1}, so that it can also be used for the integral \eqref{ric1}. 

Now by formulae \eqref{z1} and \eqref{z2} let us pass to some evaluation of elliptic integrals related to \eqref{leg1} and \eqref{ric1}:
\begin{align}\label{leg1b}
&B_8(1,\tfrac12)=\int_0^\infty\frac{{\rm d}x}{\sqrt{1+x^8}}=\sqrt{2-\sqrt2}\,\boldsymbol{K}\left({\sqrt2}-1\right)\\
&B_8(3,\tfrac12)=\int_0^\infty\frac{x^2}{\sqrt{1+x^8}}\,{\rm d}x=\sqrt{2-\sqrt2}\,\boldsymbol{K}\left({\sqrt2}-1\right)\label{ricb}\\
&A_8(5,\tfrac12)=\int_0^1\frac{x^4}{\sqrt{1-x^8}}\,{\rm d}x=\frac{\pi}{8}\frac{\sqrt2}{\boldsymbol{K}\left({\sqrt2}-1\right)}\label{qua}\\
&A_8(7,\tfrac12)=\int_0^1\frac{x^6}{\sqrt{1-x^8}}\,{\rm d}x=\frac{\pi}{24}\frac{2+\sqrt2}{\boldsymbol{K}\left({\sqrt2}-1\right)}\label{sei}
\end{align}
Thus, taking advantage of \eqref{even} in \eqref{leg1b} and \eqref{ricb}, we see that
\begin{align}
\label{fd8a}\mathrm{F}_D^{(8)}\left( \left. 
\begin{array}{c}
3;\frac12 \\[2mm]
4
\end{array}
\right|x_1,\dots,x_8  \right)&=3\sqrt{2-\sqrt2}\,\boldsymbol{K}\left({\sqrt2}-1\right)\\
\label{fd8b}\mathrm{F}_D^{(8)}\left( \left. 
\begin{array}{c}
1;\frac12 \\[2mm]
4
\end{array}
\right|x_1,\dots,x_8 \right)&=3\sqrt{2-\sqrt2}\,\boldsymbol{K}\left({\sqrt2}-1\right)
\end{align}
with $x_1,\dots,x_8$ reciprocals of the roots of $u^8+(1-u)^8=0$ as given by \eqref{repeven}, namely:
\[
\begin{split}
x_1&=1+\frac{1}{2} i
   \sqrt{2-\sqrt{2}}+\frac{\sqrt{2+\sqrt{2}}}{2},\quad x_2= 1+\frac{\sqrt{2-\sqrt{2}
   }}{2}+\frac{1}{2} i
   \sqrt{2+\sqrt{2}}\\
   x_3&=1-\frac{\sqrt{2-\sqrt{2}}}{2}+\frac{1}{2} i
   \sqrt{2+\sqrt{2}},\quad x_4=1+\frac{1}{2} i
   \sqrt{2-\sqrt{2}}-\frac{\sqrt{2+\sqrt{2}}}{2},\\
   x_5&=1-\frac{1}{2} i
   \sqrt{2-\sqrt{2}}-\frac{\sqrt{2+\sqrt{2}}}{2},\quad x_6=1-\frac{\sqrt{2-\sqrt{2}
   }}{2}-\frac{1}{2} i
   \sqrt{2+\sqrt{2}}\\
   x_7&=1+\frac{\sqrt{2-\sqrt{2}}}{2}-\frac{1}{2} i
   \sqrt{2+\sqrt{2}},\quad x_8=1-\frac{1}{2} i
   \sqrt{2-\sqrt{2}}+\frac{\sqrt{2+\sqrt{2}}}{2}
   \end{split}
\]
We can also reduce the order of the Lauricella functions in \eqref{fd8a} and \eqref{fd8b} using Lemma 1.1 of \cite{jnt2} getting special values of ${\rm F}_D^{(7)}$ with some of the arguments lying in the positive real axis:
\begin{align}
\label{fd7a}\mathrm{F}_D^{(7)}\left( \left. 
\begin{array}{c}
3;\frac12 \\[2mm]
4
\end{array}
\right|z_1,\dots,z_7\right)&=\left(-\frac{3}{\sqrt{2}}+i \left(3-\frac{3}{\sqrt{2}}\right)\right) \,\boldsymbol{K}\left({\sqrt2}-1\right)\\
\label{fd7b}\mathrm{F}_D^{(7)}\left( \left. 
\begin{array}{c}
1;\frac12 \\[2mm]
4
\end{array}
\right|z_1,\dots,z_7 \right)&=\left(3 \left(\frac{1}{\sqrt{2}}-1\right)+\frac{3 i}{\sqrt{2}}\right) \,\boldsymbol{K}\left({\sqrt2}-1\right)
\end{align}
where
\[
z_1=1-\frac{1+i}{\sqrt{2}},\quad z_2=1-i,\quad z_3=1+\frac{1-i}{\sqrt{2}},\quad z_4=2,\quad z_5=1+\frac{1+i}{\sqrt{2}},\quad z_6=1+i,\quad z_7=1-\frac{1-i}{\sqrt{2}} 
\]

For \eqref{qua} and \eqref{sei} we use \eqref{rela1} obtaining
\begin{align}
\label{fd7c}\mathrm{F}_D^{(7)}\left( \left. 
\begin{array}{c}
5;\frac12 \\[2mm]
\frac{11}{2}
\end{array}
\right|\frac{1+i}{\sqrt{2}},i,-\frac{1-i}{\sqrt{2}},-1,-\frac{1+i}{\sqrt{2}},-i,\frac{1-i}{\sqrt{2}}\right)&=\frac{315 \pi }{1024 \sqrt{2} \boldsymbol{K}\left(\sqrt{2}-1\right)}\\
\label{fd7d}\mathrm{F}_D^{(7)}\left( \left. 
\begin{array}{c}
7;\frac12 \\[2mm]
\frac{15}{2}
\end{array}
\right|\frac{1+i}{\sqrt{2}},i,-\frac{1-i}{\sqrt{2}},-1,-\frac{1+i}{\sqrt{2}},-i,\frac{1-i}{\sqrt{2}} \right)&=\frac{1001 \left(2+\sqrt{2}\right) \pi }{16384 \boldsymbol{K}\left(\sqrt{2}-1\right)}
\end{align}

Let us add to our list of integrals one more, considered by Serret, \cite{S} p. 65 and Legendre, \cite{tr-1} chapter XXX pp. 201-202, but (again) ignored nowadays:
\begin{equation}\label{serrp65}
B_6(1,\tfrac13)=\int_0^\infty\frac{{\rm d}x}{\sqrt[3]{1+x^6}}=	\frac{\sqrt[3]{4}}{\sqrt[4]{3}}\,\boldsymbol{K}\left(\frac{\sqrt6-\sqrt2}{4}\right)
\end{equation}
Invoking again \eqref{repeven} we infer
\begin{equation}\label{serret}
\mathrm{F}_D^{(6)}\left( \left. 
\begin{array}{c}
1;\frac13 \\[2mm]
2
\end{array}
\right|\,x^{(6)}_1,\,x^{(6)}_2,\,x^{(6)}_3,\,x^{(6)}_4,\,x^{(6)}_5,\,x^{(6)}_6 \right)=\frac{\sqrt[3]{4}}{\sqrt[4]{3}}\,\boldsymbol{K}\left(\frac{\sqrt6-\sqrt2}{4}\right)
\end{equation}
where
\[
\begin{split}
x^{(6)}_1&=\frac{\sqrt{3}}{2}+1+\frac{i}{2},\quad x^{(6)}_2=1+i,\quad x^{(6)}_3=1-\frac{\sqrt{3}}{2}+\frac{i}{2},\\
x^{(6)}_4&=1-\frac{\sqrt{3}}{2}-\frac{i}{2},\quad x^{(6)}_5=1-i,\quad x^{(6)}_6=\frac{\sqrt{3}}{2}+1-\frac{i}{2}
\end{split}
\]
Eventually from \eqref{repodd} we get an evaluation of the analytic continuation of $\mathrm{F}_D^{(5)}$. Putting
\[
z^{(6)}_1=\frac{1}{2}-\frac{i \sqrt{3}}{2},\quad z^{(6)}_2=\frac{3}{2}-\frac{i \sqrt{3}}{2},\quad z^{(6)}_3=2,\quad z^{(6)}_4=\frac{3}{2}+\frac{i \sqrt{3}}{2},\quad z^{(6)}_5=\frac{1}{2}+\frac{i
   \sqrt{3}}{2}
\]
we have
\begin{equation}\label{serretprol2}
\mathrm{F}_D^{(5)}\left( \left. 
\begin{array}{c}
1;\frac13 \\[2mm]
2
\end{array}
\right|\,z^{(6)}_1,\,z^{(6)}_2,\,z^{(6)}_3,\,z^{(6)}_4,\,z^{(6)}_5\right)=\frac{\sqrt[3]{4}}{\sqrt[4]{3}}\left(-\frac{\sqrt{3}}{2}+\frac{i}{2}\right)\,\boldsymbol{K}\left(\frac{\sqrt6-\sqrt2}{4}\right)
\end{equation}
In Chapter XXX section 164 pp. 205-206 of \cite{tr-1} Legendre evaluated an integral similar to \eqref{serrp65}
\begin{equation}\label{capXXX205}
A_6(1,\tfrac13)=\int_0^1\frac{{\rm d}x}{\sqrt[3]{1-x^6}}=\frac{\sqrt[3]{4}}{\sqrt[4]{27}}\,\boldsymbol{K}\left(\frac{\sqrt6-\sqrt2}{4}\right)
\end{equation}
So using \eqref{rela1} we get the identity
\begin{equation}\label{capXXX205b}
\mathrm{F}_D^{(5)}\left( \left. 
\begin{array}{c}
1;\frac13 \\[2mm]
\frac53
\end{array}
\right|\,w^{(6)}_1,\,w^{(6)}_2,\,w^{(6)}_3,\,w^{(6)}_4,\,w^{(6)}_5\right)=\frac{\sqrt[3]{32}}{\sqrt[4]{2187}}\,\boldsymbol{K}\left(\frac{\sqrt6-\sqrt2}{4}\right)
\end{equation}
where
\[
w^{(6)}_1=\frac{1}{2}+\frac{i \sqrt{3}}{2},\quad \,w^{(6)}_2=-\frac{1}{2}+\frac{i
   \sqrt{3}}{2},\quad \,w^{(6)}_3=-1,\quad \,w^{(6)}_4=-\frac{1}{2}-\frac{i \sqrt{3}}{2},\quad \,w^{(6)}_5=\frac{1}{2}-\frac{i
   \sqrt{3}}{2}
\]
Furthermore in this special situation, equation \eqref{id:3}, i.e. Lemma 1.1 of \cite{jnt2}, can also be used for an integral of the $A$ family, in \eqref{capXXX205b} yielding
\begin{equation}
\mathrm{F}_D^{(4)}\left( \left. 
\begin{array}{c}
1;\frac13 \\[2mm]
\frac53
\end{array}
\right|\frac{3}{2}+\frac{i \sqrt{3}}{2},1+i \sqrt{3},i \sqrt{3},-\frac{1}{2}+\frac{i \sqrt{3}}{2}\right)=\left(\frac12+\frac{\sqrt3}{2}i\right)\frac{\sqrt[3]{32}}{\sqrt[4]{2187}}\,\boldsymbol{K}\left(\frac{\sqrt6-\sqrt2}{4}\right)
\end{equation}
%%%%%%%%%%%%%%%%%%%%%%%%%%%%%%%%%%%%%%%%%%%%%%%%%%%%%%%%%%%%%%%%%%%%%%%%%%%%%%%%%%%%%%%%%%%%%%
\section{Hypergeometric values from reduction of some hyperelliptic integrals}

In this section we are going to compute further hypergeometric values by an altogether different technique, namely the reduction of hyperelliptic integrals to elliptic ones. 

Our first concern will be the genus of an orientable surface as preliminary to that of the hyperelliptic integrals. In \cite{Riemann1857} Riemann introduced  the genus $g$ ({\it Geschlecht}) of a surface in order to classify it: $2g+1$ is the number of closed cuts necessary for dividing it in two apart pieces. The genus of a surface characteres the connectivity of it: every closed orientable surface is topologically equivalent to a sphere with $g$ handles: the integer  $g$ is called the {\it genus} of such a surface. Thus, a sphere has genus 0, a torus genus 1, and the digit \lq\lq8'', with its two holes,  genus 2. By means of $2g$ pairs of closed cuts, a surface of genus $g$ can be transformed into a simply connected one, that is, a surface divisible into two parts by an arbitrary closed curve other than one lying on the boundary of the surface. For instance, a torus can be transformed into a simply connected surface by cuts along a meridian and a parallel. Now let us take into account the Abelian integral 
\begin{equation}\label{hypp}
\int R(z,w)\,{\rm d}z
\end{equation}
where  $R(z,w)$ is a rational function of variables $z$ and $w$, related by an algebraic equation of the type $w^2=P(z)$ where $P(z)$ is a polynomial of degree $m\geq 5$ without multiple roots. For  $m=3$ the integral \eqref{hypp} collapses to an elliptic one, while it is hyperelliptic for each $m\geq 5$. The above mentioned relationship between $w$ and $z$ corresponds to a two-sheeted compact Riemann surface of genus 
\[
g=
\begin{cases}
\dfrac{m-2}{2}\quad\text{if  $m$ is even}\\[2mm]
\dfrac{m-1}{2}\quad\text{if  $m$ is odd}
\end{cases}
\]
thus, for hyperelliptic integrals it will be the case that $g\geq 2$. Such integrals are found, for example, to provide a time for the particle motion in higher dimensional axially symmetric space–times. For some dimensions, time will be given by a hyperelliptic integral so that the dynamical problem solution will require its inversion in order to get  the instantaneous position of the particle, namely the relevant vectorial function $\overrightarrow{r} =\overrightarrow{r}(t) $. If the genus is $g = 1$, then such an inversion can result in elliptic functions, namely doubly periodic functions of one complex variable. The invertibility does depend on the nature of the involved integrals because such an inversion is possible only if they can be reduced to first kind $F(\varphi, k)$ ones. For higher {\it genera}, $g>1,$ the inversion becomes impossible because, as Jacobi recognized, $2g$-periodic functions of one variable do not exist. He solved this contradiction by formulating his famous {\it Jacobi inversion problem} that involves $g$  hyperelliptic integrals. The problem was solved
(G\"{o}pel, Rosenhain, Riemann) in terms of so called {\it hyperelliptic functions} which are indeed $2g$-periodic functions for $g>1$ which depend on $g$ variables while the periods (also called {\it moduli}) are $g \times g$--matrices. The domain of such functions, the Jacobi variety, is thus the $g$-dimensional complex space $C_{g}$ quotiented by the period lattice.
The Jacobi inversion problem  helped to solve some mechanical systems such as the Neumann geodesic on an ellipsoid, the spinning top of Kovalewskaja, the Kirchhoff motion of a rigid body in a fluid, and many others. This special type of integrability, called {\it algebro--geometric integrability}, has been receiving much attention owing to the discovery of many partial differential equations of Korteveg--de Vries type with connections to soliton theory.

The above concept of genus will be kept apart from the genus $g^{*}$ of a curve introduced by A. Clebsch \cite{Clebsch1865} 
which is a number characterizing an algebraic curve $f(x, y)$ of the $m{\rm th}$ degree
$$
g^{*}=\frac{(m-1)(m-2)}{2}-\delta
$$
as the difference between the highest possible number of double points an $m{\rm th}$ order curve can have and the actual number of its double points.% obtainable by the system built by taking to zero the first partial derivatives of  $f(x, y)$.
 When more complex singular points are present, they are counted as the corresponding number of double points; for example, a {\it cusp} is counted as one double point, and a triple point is counted as three.
% Order, class and genus are the {\it Pl\"{u}ckerian characters} of a planar algebraic curve. 
 
 Genus, as spoken of, means the highest possible number of singularities for an $n{\rm th}$ order algebraic curve: for instance a planar cubic without  double points has genus 1; with only one double point, then $g^{*}=0$. Second-order curves are of genus 0. All the curves having $g^{*}=0$ are called {\it unicursal}, due to its being possible to
assign their coordinates as  rational functions of a parameter.  A beautiful theorem of Clebsch (1864) established that  the coordinates of any curve of genus 1 can be given as Jacobian elliptic functions of a parameter.
 Third-order curves can be of genus 0 or 1. For example for $y -x^{3}= 0 $ we have $g^{*}=1$. On the other hand, the semicubical parabola $y^{2} - x^{3} = 0$, characterized by one cusp will have $g^{*}=0$. The same is true for the {\it folium of Descartes} $x^{3}+ y^{3} -3axy = 0$, which has one double point.

In our previous papers, \cite{jnt1,jnt2} we started with the employ of some famous reduction schemes of hyperelliptic integrals to elliptic in order to evaluate special values of hypergeometric functions arising from the integral representation theorem, \eqref{irto}.
For the purpose let us recall two of our previous identities.
First, formula (4.5) of \cite{jnt1}, theorem 4.3, which, by means of Jacobi reduction  \cite{cre}, changes to elliptic a genus 2 hyperelliptic integral:
\begin{equation}\label{jac}
\int_{0}^{1}\frac{\left( \sqrt{ab}+z \right) \,{\rm d}z }{\sqrt{z
(z -1)(z -ab)(z -a)(z -b)}}=\frac{1}{\sqrt{(1-a)(1-b)}}%
\int_{0}^{1}\frac{{\rm d}x }{\sqrt{x (1-x )(1-cx )}}%=\frac{2}{\sqrt{ab}-1}
%\,\boldsymbol{K}\left( \frac{\sqrt{a}-\sqrt{b}}{\sqrt{ab}-1}\right).  
\end{equation}
where $a>b>1$ and
\begin{equation*}
c=-\frac{\left( \sqrt{a}-\sqrt{b}\right) ^{2}}{(1-a)(1-b)}.  \label{ci}
\end{equation*}
Formula \eqref{jac} stems from the second degree transformation:
\[
x =\frac{(1-a)(1-b)z }{(z -a)(z -b)}.
\]
Additionally, let us quote identity (35), theorem 4.1 of \cite{jnt2} founded on the first Hermite reduction \cite{her}, linking again a hyperelliptic integral of genus 2 to an elliptic integral:
\begin{equation}
\int_{z_{1}}^{\infty }\frac{z}{\sqrt{(z^{2}-a)(4z^{3}-3az-b)}}\,
\mathrm{d}z=\frac{1}{\sqrt{6}}\int_{-2z_{1}}^{\infty }\frac{\mathrm{d}y}{\sqrt{y^{3}-3ay+2b}}.  \label{ugu}
\end{equation}
Formula \eqref{ugu} comes from a change of third degree:
\[
y=\frac{2(z^{3}-b)}{3(z^{2}-a)}
\]
$z_1$ being a root of $4z^{3}-3az-b=0$ while $2z_1$ solves $y^{3}-3ay+2b=0.$

Hereafter we go on with the review of reductions including formulae due to Goursat \cite{goursat}, Hermite \cite{her}, considering further reductions of hyperelliptic integrals of a genus 2 and 3, the last being reduced through a sixth degree transformation of variables, \cite{herovres}.
Then we will consider yet higher order ones, founded upon contributions of Belokolos et al. \cite{Belokolos1986}, Eilbeck and Enol'skii \cite{eilbeck1994elliptic}, Enol'skii and Kostov \cite{enol1994geometry} and  Maier \cite{maier2008}.

\subsection{Goursat reduction: genus 2}
In a 1885 article \cite{goursat}, Goursat found a reduction formula for a hyperelliptic integral of genus $g=2$. Through the change of variables:
\[
x=\frac{t^3+a t+b}{3 t-p}
\]
starting from an elliptic integral, he obtained the reduction formula:
\begin{equation}\label{redg}
\begin{split}
&\int\frac{{\rm d}x}{\sqrt{x \left(4 (3 x-a)^3-27 (b+p x)^2\right)}}=\\
&\int\frac{-a p-3 b-3 t^2 (p-2 t)}{\sqrt{\left(a t+b+t^3\right) \left(a p+3 b+3 t^2 (p-2 t)\right)^2 \left(4 (a
   p+3 b)+3 p t^2+3 t^3\right)}} \,{\rm d}t
   \end{split}
\end{equation}
Once a suitable interval of integration is chosen, it is possible to remove the sign uncertainty and simplify \eqref{redg} obtaining
\begin{equation}\label{redg1}%\tag{\ref{redg}a}
\int\frac{{\rm d}x}{\sqrt{x \left(4 (3 x-a)^3-27 (b+p x)^2\right)}}=\pm\int\frac{{\rm d}t}{\sqrt{(a t+b+t^3)( 4 (a p+3 b)+3 p t^2+3 t^3)}}.
\end{equation}
If we make the simple choice to take $a=p=0$ and to integrate in $[(b^2/4)^{1/3},\infty),$ on the left hand side of \eqref{redg1} we get
\begin{equation}\label{redg2}%\tag{\ref{redg}b}
\int_{\left(\frac{b^2}{4}\right)^{1/3}}^\infty\frac{{\rm d}x}{\sqrt{27x(4x^3-b^2)}}=\int_{\left(\frac{b}{2}\right)^{1/3}}^\infty\frac{{\rm d}t}{\sqrt{3(t^3+b)(t^3+4b)}}.
\end{equation}
Normalizing the interval of integration, we see that the most natural choice $b=2$ is not restrictive, so we can use, as starting point to find our next identity, the reduction:
\begin{equation}\label{dig}
\int_1^\infty\frac{{\rm d}x}{\sqrt{x(x^3-1)}}=6\int_1^\infty\frac{{\rm d}t}{\sqrt{(t^3+2)(t^3+8)}}.
\end{equation}
The hyperelliptic integral on the right hand side of \eqref{dig} is a particular case of the family of integrals described by the following lemma.
\begin{lemma}\label{lle1}
Let $2n-m>0,\,a,\,b>0$ then
\begin{equation}\label{lem1}
\int_1^\infty\frac{{\rm d}t}{\sqrt[m]{(t^n+a)(t^n+b)}}=\frac{m}{2n-m}\,\mathrm{F}_{1}\left( \left. 
\begin{array}{c}
\frac{2n-m}{mn};\frac{1}{m},\frac{1}{m} \\[2mm]
\frac{2n-m+mn}{mn}
\end{array}
\right| -a,-b\right).
\end{equation}
\end{lemma} 
\begin{proof}
Condition $2n-m>0$ ensures the convergence of the integral at the right hand side of \eqref{lem1}. Thesis \eqref{lem1} follows immediately from the integral representation theorem for the Appell ${\rm F}_1$ after the change of variables $t=1/u^{1/n}$.
\end{proof}
From \eqref{dig} and Lemma \ref{lle1} we infer the following new evaluation of the analytic continuation of ${\rm F}_1.$
\begin{theorem}\label{Gou0}
\begin{equation}\label{bg00}
\mathrm{F}_{1}\left( \left. 
\begin{array}{c}
\frac23;\frac12,\frac12 \\[2mm]
\frac53
\end{array}
\right| -2,-8\right)=\frac{1}{\pi\sqrt[3]{16}\sqrt{3}}\,\Gamma^3\left(1/3\right).
\end{equation}
\end{theorem} 
\begin{proof}
Using Lemma \ref{lle1} in equation \eqref{dig} we see that
\begin{equation}\label{dig1}%\tag{\ref{dig}a}
\int_1^\infty\frac{{\rm d}x}{\sqrt{x(x^3-1)}}=3\mathrm{F}_{1}\left( \left. 
\begin{array}{c}
\frac23;\frac12,\frac12 \\[2mm]
\frac53
\end{array}
\right| -2,-8\right)
\end{equation}
But we also have
\begin{equation}\label{bet}
\int_1^\infty\frac{{\rm d}x}{\sqrt{x(x^3-1)}}=\int_0^1\frac{{\rm d}x}{\sqrt{1-x^3}}=\frac{\sqrt\pi}{3}\,\frac{\Gamma(1/3)}{\Gamma(5/6)}
\end{equation}
Thesis \eqref{bg00} follows by equating \eqref{dig1} and \eqref{bet} and using the Euler reflection formula.
\end{proof}
\begin{remark}
If, instead of using the Euler function to evaluate the integral in \eqref{bet}, we employ the elliptic integral of first kind, by \cite{by} entry 244.00 p. 92 we get: 
\begin{equation}\tag{\ref{bg00}a}
\mathrm{F}_{1}\left( \left. 
\begin{array}{c}
\frac23;\frac12,\frac12 \\[2mm]
\frac53
\end{array}
\right| -2,-8\right)=\frac{1}{3^{5/4}}\,F\left(\arccos(2-\sqrt3),\frac{\sqrt6+\sqrt2}{4}\right).
\end{equation}

\end{remark}

In order to evaluate the Appell ${\rm F}_1$ not in its analytic continuation, but just in the convergence domain of the double power series which defines it, we present a slight modification of the Goursat reduction formula,
\begin{equation}\label{gb0}
\int_0^1\frac{{\rm d}x}{\sqrt{1-x^3}}=\int_0^1\frac{6}{\sqrt{(t^3+2)(t^3+8)}}{\rm d}t.
\end{equation} 
Identity \eqref{gb0} comes from the cubic change of variable
\[
x=\frac{3t}{t^3+2}.
\]
The hyperelliptic integral at the right hand side of \eqref{gb0} is computable through the Appell ${\rm F}_1$ hypergeometric function of two variables. In fact we have, in analogy with Lemma \ref{lle1}:
\begin{lemma}\label{llem2}
Let $a,\,b>0$ then
\begin{equation}\label{lem2}
\int_0^1\frac{{\rm d}t}{\sqrt[m]{(t^n+a)(t^n+b)}}=\frac{1}{\sqrt[m]{ab}}\,
\mathrm{F}_{1}\left( \left. 
\begin{array}{c}
\frac{1}{n};\frac{1}{m},\frac{1}{m} \\[2mm]
\frac{1}{n}+1
\end{array}
\right| -\frac{1}{a},-\frac{1}{b}\right).
\end{equation}

\end{lemma}
From Lemma \ref{llem2} we get
\begin{equation}\label{bg011}
\int_0^1\frac{{\rm d}t}{\sqrt{(t^3+2)(t^3+8)}}=\frac14\,\mathrm{F}_{1}\left( \left. 
\begin{array}{c}
\frac13;\frac12,\frac12 \\[2mm]
\frac43
\end{array}
\right| -\frac12,-\frac18\right).
\end{equation}
We can now state a further identity involving $\pi$ which concerns a new evaluation of the Appell ${\rm F}_1$.
\begin{theorem}\label{Gou1}
\begin{equation}\label{bg01}
\mathrm{F}_{1}\left( \left. 
\begin{array}{c}
\frac13;\frac12,\frac12 \\[2mm]
\frac43
\end{array}
\right| -\frac12,-\frac18\right)=\frac{1}{\pi\sqrt{27}\sqrt[3]{2}}\,\Gamma^3\left(\frac13\right).
\end{equation}
\end{theorem}
\begin{proof}
First, let us evaluate the elliptic integral at the left hand side of \eqref{gb0} using the Euler Beta function:
\begin{equation}\label{bg012}
\int_0^1\frac{{\rm d}x}{\sqrt{1-x^3}}=\frac{\sqrt\pi}{3}\frac{\Gamma(1/3)}{\Gamma(5/6)}.
\end{equation}
Thesis \eqref{bg01} follows from \eqref{gb0} evaluating the integrals therein using \eqref{bg011} (right hand side) and \eqref{bg012} (left hand side) and again invoking the Euler reflection formula.
\end{proof}
\begin{remark}
If, instead of using the Euler Beta function, we evaluate the elliptic integral at the left hand side of \eqref{gb0} using entry 244.00 p. 92 of \cite{by}, instead of \eqref{bg012}, we will get:
\begin{equation}%\tag{\ref{bg012}b}
\int_0^1\frac{{\rm d}x}{\sqrt{1-x^3}}=\frac{1}{\sqrt[4]{3}}\,F\left(\arccos(2-\sqrt3),\frac{\sqrt6+\sqrt2}{4}\right)
\end{equation}
leading to the elliptic analogue of \eqref{bg01}:
\begin{equation}\label{5.9b}%\tag{\ref{bg01}b}
\mathrm{F}_{1}\left( \left. 
\begin{array}{c}
\frac13;\frac12,\frac12 \\[2mm]
\frac43
\end{array}
\right| -\frac12,-\frac18\right)=\frac{2}{3\sqrt[4]{3}}\,F\left(\arccos(2-\sqrt3),\frac{\sqrt6+\sqrt2}{4}\right).
\end{equation}
\end{remark}

The family of integral identities produced by the Goursat transformation is huge. Nevertheless the evaluation is intricate when some of the cubics in \eqref{redg1} are complete. Taking for instance $a=0,\,b=p=1$ in \eqref{redg1} we get the identity
\begin{equation}\label{011}
\int_1^\infty\frac{{\rm d}x}{\sqrt{x(x-1)\left(4 x^2+3 x+1\right)}}=\int_1^\infty\frac{3}{\sqrt{\left(t^3+1\right) \left(t^3+t^2+4\right)}}\,{\rm d}t
\end{equation} 
which can be rewritten as
\begin{equation}\label{011b}%\tag{\ref{011}b}
\int_0^1\frac{{\rm d}x}{\sqrt{(1-x)\left(x^2+3 x+4\right)}}=\int_0^1\frac{3t}{\sqrt{\left(t^3+1\right) \left(4t^3+t+1\right)}}\,{\rm d}t.
\end{equation}
We need to employ the Lauricella ${\rm F}^{(6)}_D$ in order to evaluate the hyperelliptic integral at right hand side of \eqref{011b}. To do it, first observe that:
\[
\begin{split}
t^3+1&=(1+t) \left(1-\left(\frac{1}{2}-\frac{i \sqrt{3}}{2}\right) t\right)
   \left(1-\left(\frac{1}{2}+\frac{i \sqrt{3}}{2}\right) t\right)=(1+t)(1-\omega t)(1-\overline{\omega}t)\\
4t^3+t+1&=(1+2 t)  \left(1-\left(\frac{1}{2}-\frac{i
   \sqrt{7}}{2}\right) t\right) \left(1-\left(\frac{1}{2}+\frac{i \sqrt{7}}{2}\right) t\right)=(1+2t)(1-\varepsilon t)(1-\overline{\varepsilon}t)   
   \end{split}
\]
hence
\begin{equation}\label{011c}\tag{\ref{011}c}
\int_0^1\frac{t}{\sqrt{\left(t^3+1\right) \left(4t^3+t+1\right)}}\,{\rm d}t=\frac12\,{\rm F}^{(6)}_D\left( \left. 
\begin{array}{c}
2;\frac12\\[2mm]
3
\end{array}
\right| -1,\omega,\overline{\omega},-2,\varepsilon,\overline{\varepsilon}\right)
\end{equation}
We follow the same strategy to evaluate the integral at the left hand side of \eqref{011b} starting from the factorization
\[
x^2+3 x+4=4 \left(1-\left(-\frac{3}{8}-\frac{i \sqrt{7}}{8}\right) x\right) \left(1-\left(-\frac{3}{8}+\frac{i
   \sqrt{7}}{8}\right) x\right)=4(1-\alpha x)(1-\overline{\alpha}x)
\]
leading to the integration formula:
\begin{equation}\label{011d}\tag{\ref{011}d}
\int_0^1\frac{{\rm d}x}{\sqrt{(1-x)\left(x^2+3 x+4\right)}}={\rm F}_1\left( \left. 
\begin{array}{c}
1;\frac12,\frac12 \\[2mm]
\frac{3}{2}
\end{array}
\right| \alpha,\overline{\alpha}\right).
\end{equation}
In such a way we have proved the following reduction formula from $ {\rm F}^{(6)}_D $ to ${\rm F}_1 $ :
\begin{theorem}
Defining:
\[
\alpha=-\frac{3}{8}-\frac{i \sqrt{7}}{8},\quad
\omega=\frac{1}{2}-\frac{i \sqrt{3}}{2},\quad
\varepsilon=\frac{1}{2}-\frac{i\sqrt{7}}{2}
\]
We have the reduction:
\begin{equation}\label{laured}
{\rm F}^{(6)}_D\left( \left. 
\begin{array}{c}
2;\frac12\\[2mm]
3
\end{array}
\right| -1,\omega,\overline{\omega},-2,\varepsilon,\overline{\varepsilon}\right)=\frac23\,{\rm F}_1\left( \left. 
\begin{array}{c}
1;\frac12,\frac12 \\[2mm]
\frac{3}{2}
\end{array}
\right| \alpha,\overline{\alpha}\right).
\end{equation}
\end{theorem}
\begin{remark}
Of course, in evaluating the elliptic integral in \eqref{011d} we can follow the elliptic way, \cite{by} entry 243.00 p. 91:
\begin{equation}\label{24300}
\int_0^1\frac{{\rm d}x}{\sqrt{(1-x)\left(x^2+3 x+4\right)}}=\frac{1}{2^{3/4}}\,F\left(\arccos\left(\frac{1}{7} \left(9-4 \sqrt{2}\right)\right),\frac{1}{4}
   \sqrt{8+5 \sqrt{2}}\right).
\end{equation}
So, comparing \eqref{011b} and \eqref{24300} we get
\begin{equation}\label{lauredb}
{\rm F}^{(6)}_D\left( \left. 
\begin{array}{c}
2;\frac12\\[2mm]
3
\end{array}
\right| -1,\omega,\overline{\omega},-2,\varepsilon,\overline{\varepsilon}\right)=\frac{\sqrt[4]{2}}{3}\,F\left(\arccos\left(\frac{1}{7} \left(9-4 \sqrt{2}\right)\right),\frac{1}{4}
   \sqrt{8+5 \sqrt{2}}\right).
\end{equation}
\end{remark}
\begin{remark}
Notice that the Lauricella ${\rm F}^{(6)}_D$ appearing in \eqref{laured} and \eqref{lauredb} fulfills hypotheses of Lemma 1.1 of \cite{jnt2}; thus in  \eqref{laured} and \eqref{lauredb} the left hand side can be replaced with a Lauricella of lower order, namely:
\[
\frac14\,{\rm F}^{(5)}_D\left( \left. 
\begin{array}{c}
2;\frac12 \\[2mm]
3
\end{array}
\right| -\frac12,\frac{3-i\sqrt3}{4},\frac{3+i\sqrt3}{4},\frac{3-i\sqrt7}{4},\frac{3+i\sqrt7}{4}\right)
\]
\end{remark}
\subsection{Hermite reduction: genus 2}
In a paper of 1876 \cite{her} Hermite proposed a further formula reducing a genus 2 hyperelliptic integral:
\begin{equation}\label{her1876}
\int\frac{{\rm d}z}{\sqrt{(z^2-a)(8z^3-6az-b)}}=\frac13\int\frac{{\rm d}x}{\sqrt{(2ax-b)(x^2-a)}}
\end{equation}
thanks to a cubic transformation $x=(4z^3-3az)/a$.
The use of \eqref{her1876} keeping free both parameters $a$ and $b$ would allow a heavy discussion on the placing of roots of $8z^3-6az-b=0$ and $(2ax-b)(x^2-a)=0$ and Cardano's formulae for them.
Hence we prefer to study the implications of \eqref{her1876} restricting to easy choices of $a$ and $b$. This will lead to less computations and to an interesting identity linking the Appell function to the complete elliptic integral of  modulus $1/\sqrt2.$

In \eqref{her1876} we  take $b=0,$ $a^2$ instead of $a$ assuming $a>0$ and in the elliptic integral we take the interval $[0, a/2]$. In such a way \eqref{her1876} reads as:
\begin{equation}\label{her1876b0}
3a\int_0^{\frac{a}{2}}\frac{{\rm d}z}{\sqrt{z(z^2-a^2)(4z^2-3a^2)}}=\int_{-a}^0\frac{{\rm d}x}{\sqrt{x(x^2-a^2)}}.
\end{equation}
Equation \eqref{her1876b0} leads to:
\begin{theorem}
\begin{equation}\label{hermyF1}
\boldsymbol{K}\left(\frac{1}{\sqrt2}\right)=\sqrt3 \,{\rm F}_1\left( \left. 
\begin{array}{c}
\frac14;\frac12,\frac12 \\[2mm]
\frac54
\end{array}
\right| \frac13,\frac14\right).
\end{equation}
\end{theorem}
\begin{proof}
On the left hand side of \eqref{her1876b0} let us do some elementary changes while for the right hand side we will refer to the entry 234.00 of \cite{by} so that we get:
\begin{equation}\label{hermypF1}
\frac{3}{2\sqrt{2a}}\int_0^1\frac{v^{-3/4}}{\sqrt{(1-\frac13v)(1-\frac14 v)}}{\rm d}v=\frac{\sqrt2}{\sqrt{a}}\boldsymbol{K}\left(\frac{1}{\sqrt2}\right)
\end{equation}
Thesis \eqref{hermyF1} now follows from the integral representation theorem \eqref{F1}.
\end{proof}
Let us provide an application of \eqref{her1876} with both parameters nonzero: we take $a=28/3$ and $b=-48$  along the integration interval $[-2\sqrt7/\sqrt3,-18/7].$ Then the cubic on the left hand side of \eqref{her1876} has rational roots, the elliptic integral at left hand side is complete and \eqref{her1876} reads as
\begin{equation}\label{her1876bn0}
\int_1^{\frac{\sqrt7}{\sqrt3}}\frac{{\rm d}z}{\sqrt{\left(z^2-\frac{28}{3}\right)(z-1)(z-2)(z+3)}}{\rm d}z=\frac{1}{\sqrt{21}}\int_{-2\frac{\sqrt7}{\sqrt3}}^{-\frac{18}{7}}\frac{{\rm d}x}{\sqrt{\left(x+\frac{18}{7}\right)\left(x^2-\frac{28}{3}\right)}}.
\end{equation}
Then we can state:
\begin{theorem}
\begin{equation}\label{her1876bth}
\frac{5}{14} \sqrt{\frac{7}{3}+\sqrt{\frac{7}{3}}}\,\boldsymbol{K}\left(\frac{1}{7} \sqrt{\frac{1}{2} \left(49-9 \sqrt{21}\right)}\right)= \,{\rm F}_{D}^{(4)}\left( \left. 
\begin{array}{c}
\frac12;\frac12\\[2mm]
\frac34
\end{array}
\right|\frac{3 \sqrt{21}-17}{25},\frac{3-\sqrt{21}}{12}
  ,\frac{\sqrt{21}-3}{3}
  ,\frac{11-\sqrt{21}}{25}
 \right).
\end{equation}
\end{theorem}
\begin{proof} The hypergeometric approach through the integral representation theorem \ref{irto} makes valuable the integral at left hand side of \eqref{her1876bn0}. In fact if $a<b<c<y<d<e$ , passing to $z=c - u (c - y)$  we get:
\begin{equation}\label{irto1876}
\int_c^y\frac{{\rm d}z}{\sqrt{(z-a)(z-b)(z-c)(z-d)(z-e)}}=2 \sqrt{\frac{y-c}{(c-a) (c-b) (d-c) (e-c)}} \,\boldsymbol{F_4}
\end{equation}
where, following for simplicity the convention of the repeated parameter, we put:
\[
\boldsymbol{F_4}=\mathrm{F}_D^{(4)}\left( \left. 
\begin{array}{c}
\frac12;\frac12 \\[2mm]
\frac32
\end{array}
\right|\frac{c-y}{c-a},\frac{c-y}{c-b},-\frac{c-y}{d-c},-\frac{c-y}{e-
   c}\right).
\]
The integral on the left hand side of \eqref{her1876bn0} is then evaluated  by identification:
\[
\begin{matrix} a & b & c & y & d & e\\ \updownarrow & \updownarrow &
\updownarrow & \updownarrow & \updownarrow & \updownarrow\\ -2\sqrt7/\sqrt3 & -3 & 1 & \sqrt7/\sqrt3 & 2 & 2\sqrt7/\sqrt3\end{matrix} 
\]
On the other side the integral on the right hand side of \eqref{her1876bn0} is a complete elliptic integral of the first kind, as shown in \cite{by} entry 234.00. Thesis \eqref{her1876bth} follows by comparison.
\end{proof}
\subsection{Hermite reduction: genus 3}
In the Comptes Rendus de l'Academie des Sciences, Hermite introduced the transformation, included in the third volume of his works, \cite{herovres} p. 399, reducing a hyperelliptic integral of genus 3 to an elliptic integral through the sixth degree change:
\[
y=\frac{\psi(x)}{12x(x^2-a)^2}\\
\]
where
\[
\begin{split}
\psi(x)&=\phi(x)-12x(x^2-a)(10x^3-8ax-b)\\
\phi(x)&=125x^6-210ax^4-22bx^3+93a^2x^2+18abx+b^2-4a^3
\end{split}
\]
We find
\begin{equation}\label{heovr}
\int\frac{{\rm d}y}{\sqrt{y^3-3ay+b}}=2\sqrt3\int\frac{5x^2-a}{\sqrt{x\phi(x)}}\,{\rm d}x.
\end{equation}
A new identity linking  Lauricella functions to complete elliptic integrals in \eqref{heovr} can be found by fixing $a=1,\,b=0$ and integrating with respect to $y$ in $[-\sqrt3,0]$, so that \eqref{heovr} specializes as
\begin{equation}\label{heovr1}
\sqrt[4]{\frac43}\boldsymbol{K}\left(\frac{1}{\sqrt2}\right)=\frac{2}{5}\sqrt{\frac35}\int_{\frac{\sqrt{11}-\sqrt{3}}{2}}^{\frac{2}{\sqrt{5}}}\frac{5x^2-1}{\sqrt{x\left(\frac45-x^2\right)\left[x^2-\left(\frac{\sqrt8-\sqrt3}{5}\right)^2\right]\left[\left(\frac{\sqrt8+\sqrt3}{5}\right)^2-x^2\right]}}{\rm d}x.
\end{equation}
For the integral on right hand side of \eqref{heovr1} a formula will be used stemming from the integral representation theorem, \eqref{irto}, where $0<a<y<b<c$:
\begin{equation}\label{irtoapp}
\int_y^b\frac{x^m}{\sqrt{x(b^2-x^2)(x^2-a^2)(c^2-x^2)}}{\rm d}x=y^{m-\frac32}\sqrt{\frac{b^2-y^2}{(y^2-a^2)(c^2-y^2)}}\,\boldsymbol{F_3}
\end{equation}
where
\[
\boldsymbol{F_3}=\mathrm{F}_D^{(3)}\left( \left. 
\begin{array}{c}
\frac12;\frac34-\frac{m}{2},\frac12,\frac12 \\[2mm]
\frac32
\end{array}
\right|-\frac{b^2-y^2}{y^2},\frac{b^2-y^2}{a^2-y^2},\frac{b^2-y^2}{c^2-y^2}\right).
\]
Then by \eqref{heovr1} and \eqref{irtoapp} by identification of parameters we get:
\begin{theorem}
\begin{equation}\label{idhermiteK}
\begin{split}
\int_{-\sqrt{3}}^0\frac{{\rm d}y}{\sqrt{y^3-3y}}&=\sqrt[4]{\frac{4}{3}}\,\boldsymbol{K}\left(\frac{1}{\sqrt2}\right)\\
&=\sqrt{\frac{12}{125}}\left(\sqrt[4]{\frac{15625 \left(1649+225 \sqrt{33}\right)}{55296}}H_1-\frac{1}{96} \sqrt{5 \left(1552 \sqrt{3}+816 \sqrt{11}\right)}\,H_2\right)
\end{split}
\end{equation}
where $H_1$ and $H_2$ are two Lauricella ${\rm F}_D^{(3)}$ functions:
\[
\begin{split}
H_1&={\rm F}^{(3)}_D\left( \left. 
\begin{array}{c}
1;-\frac14,\frac12,\frac12 \\[2mm]
\frac32
\end{array}
\right| \frac{3-\sqrt{33}}{10},\frac{5 \left(27-5 \sqrt{33}\right)}{153+8 \sqrt{6}-25 \sqrt{33}},\frac{5 \left(27-5 \sqrt{33}\right)}{153-8 \sqrt{6}-25 \sqrt{33}}\right)\\
H_2&={\rm F}^{(3)}_D\left( \left. 
\begin{array}{c}
1;\frac34,\frac12,\frac12 \\[2mm]
\frac32
\end{array}
\right| \frac{3-\sqrt{33}}{10},\frac{5 \left(27-5 \sqrt{33}\right)}{153+8 \sqrt{6}-25 \sqrt{33}},\frac{5 \left(27-5 \sqrt{33}\right)}{153-8 \sqrt{6}-25 \sqrt{33}}\right).
\end{split}
\] 
\end{theorem}
\subsection{Transformation of integrals of genus $\boldsymbol{g=4}$}
In last twenty years  Belokolos et al. \cite{Belokolos1986}, Eilbeck and Enol'skii \cite{eilbeck1994elliptic}, Enol'skii and Kostov \cite{enol1994geometry}, Maier \cite{maier2008}, moved by ordinary differential equations of mathematical physics, studied how to reduce some higher genus hyperelliptic integrals: the relevant best synthesis is probably the Theorem 4.1 of \cite{maier2008}.
 Restricting ourselves to the  integral where the cubic $4x^3-12x$ appears, in such a way following the \cite{maier2008} notation, we get: 
 $e_\gamma=0,\,g_2=12,\,g_3=0.$ 
 Then, developing in the case $\ell=4$ the change of variable, equation (4.5) Theorem 4.1 of \cite{maier2008}, we find the following 10${\rm th}$ degree transformation:
\[
y=\frac{\left(x^2-84\right) \left(x^4+1617 x^2-1333584\right)^2}{100 \left(x^3-1029 x\right)^2 \left(x^3-624 x\right)}
\]
and then the reduction scheme:
\begin{equation}\label{ell4}
\begin{split}
\int_{-\sqrt{3}}^0\frac{{\rm d}y}{\sqrt{y^3-3y}}&=\sqrt[4]{\frac{4}{3}}\,\boldsymbol{K}\left(\frac{1}{\sqrt2}\right)\\
&=\int_{-2\sqrt{21}}^{5 \sqrt{3}-2 \sqrt{66}}\frac{10 \sqrt{x} \left(273-x^2\right)}{\sqrt{\left(x^2-624\right)
   \left(x^2-84\right) \left(x^4-678 x^2+35721\right)}}{\rm d}x.
\end{split}
\end{equation}
Next, by changing variable at right hand side in \eqref{ell4} we find:
\begin{equation}\label{ell4b}
\sqrt[4]{\frac{4}{3}}\,\boldsymbol{K}\left(\frac{1}{\sqrt2}\right)=\int_{3 \left(113-20 \sqrt{22}\right)}^{84}\frac{5 (273-u)u^{-1/4}}{ \sqrt{(624-u) (u-84) \left(u-3 \left(113-20 \sqrt{22}\right)\right) \left(u-3 \left(113+20 \sqrt{22}\right)\right)}}\,{\rm d}u.
\end{equation}
The integral at the right hand side of \eqref{ell4b} can be computed through the following  formula, where $0<a<b<c<d$, stemming from \eqref{irto} too:
\begin{equation}\label{irtg4}
\int_a^b\frac{x^m}{\sqrt{(x-a)(x-b)(x-c)(d-x)}}\,{\rm d}x=
\frac{\pi\,a^m}{\sqrt{(c-a)(d-a)}}\,\boldsymbol{F_{31}}
\end{equation}
where
\[
\boldsymbol{F_{31}}=\mathrm{F}_D^{(3)}\left( \left. 
\begin{array}{c}
\frac12;-m,\frac12,\frac12 \\[2mm]
1
\end{array}
\right|\frac{a-b}{a},\frac{b-a}{c-a},\frac{b-a}{d-a}\right)
\]
Then thanks to formulae \eqref{irtg4} and \eqref{ell4b} we have:
\begin{theorem}\label{genus4}
\begin{equation}
\sqrt[4]{\frac{4}{3}}\,\boldsymbol{K}\left(\frac{1}{\sqrt2}\right)=\frac{91 \pi }{6 \sqrt[4]{264 \left(169+36 \sqrt{22}\right)}} L_1-\frac{\left(113-20 \sqrt{22}\right)^{3/4} \pi }{2 \sqrt[4]{3}
   \sqrt{176+38 \sqrt{22}}}L_2
   \end{equation}
\end{theorem}
where $L_1$ and $L_2$ are two Lauricella ${\rm F}_D^{(3)}$ functions:
\[
\begin{split}
L_1&={\rm F}^{(3)}_D\left( \left. 
\begin{array}{c}
\frac12;\frac14,\frac12,\frac12 \\[2mm]
1
\end{array}
\right| \frac{5}{567} \left(23-16
   \sqrt{22}\right),\frac{1}{2}-\frac{17}{8 \sqrt{22}},16
   \sqrt{22}-75\right)\\
L_2&={\rm F}^{(3)}_D\left( \left. 
\begin{array}{c}
\frac12;-\frac34,\frac12,\frac12 \\[2mm]
1
\end{array}
\right| \frac{5}{567} \left(23-16
   \sqrt{22}\right),\frac{1}{2}-\frac{17}{8 \sqrt{22}},16
   \sqrt{22}-75\right).
\end{split}
\] 
Although it is an elementary consequence of the just proved theorem, we deem it of some interest to carry out the following new formula for $\pi$ where $L_1$ and $L_2$ are the same as in theorem \ref{genus4}
\begin{cor}\label{genus4cor}
\[
\pi=\frac{12 \sqrt{22 \left(9+2 \sqrt{22}\right)}}{91 L_1\sqrt[4]{22} -L_2\sqrt{21569 \sqrt{22}-99440} }\,\boldsymbol{K}\left(\frac{1}{\sqrt2}\right)
\]
\end{cor}
\subsection{Transformation of integrals of genus $\boldsymbol{g>4}$}
It is rather natural to try to go beyond, specially after the papers \cite{Belokolos1986, eilbeck1994elliptic, enol1994geometry,maier2008} where the possibility is hinted at of reducing hyperelliptic integrals of any genus. Unfortunately such reductions are not of practical use, due to the need to work with
definite integrals. For instance, following  theorem 4.1 of \cite{maier2008} with genus $g=5$ the change of variable transforming the complete first kind elliptic integral\[
\int_{-\sqrt{3}}^0\frac{{\rm d}y}{\sqrt{y^3-3y}}=\sqrt[4]{\frac{4}{3}}\,\boldsymbol{K}\left(\frac{1}{\sqrt2}\right)
\]
into a hyperelliptic integral $g=5$ is:
\[
y=\frac{\left(x^3-1584 x\right) \left(x^6+10242 x^4-21775959
   x^2+7971615000\right)^2}{225 \left(x^2-324\right) \left(x^6-5382
   x^4+ 4922937x^2+1771470000\right)^2}
\]
so that to compute the integration limits of the hyperelliptic integral is just outside of practicability.  The reason is that when applying the change of variables to the upper and lower limits of
integration, one must evaluate algebraic functions that are not expressible in terms of radicals.  That is, one must solve
polynomial equations which may be quintic, sextic, etc. which are in general not solvable by radicals.

\section{Conclusions}
In section 3 an autonomous and elementary proof of the Kummer  identity \eqref{ikummer} was established. 
Besides, this paper holds the proofs of a set of formulae providing the values at some special points, both inside/outside the unit disk, of multivariate hypergeometric  functions.
Among them we can conclusively cite e. g. the following evaluations.
First \eqref{even} evaluates a Lauricella of order $2m$;  \eqref{fd7b}--\eqref{fd7c} 
evaluate Lauricellas of index 7 through a complete elliptic integral of the
first kind; \eqref{bg00} evaluates an Appell function through values of Gamma; equation \eqref{5.9b} evaluates an Appell function again through incomplete elliptic integrals of the first kind;  and so on. The chief tools were the usual representation for such functions and the \textit{double evaluation method,} or established schemes for the reduction of hyperelliptic integrals.
Several new formulae were presented like that of Theorem \ref{genus4} and the Corollary \ref{genus4cor} evaluating $\pi$.  Finally, hyperelliptic integrals of increasing genus have been reduced via the classic reduction schemes and the more recent ones introduced in \cite{Belokolos1986}, \cite{eilbeck1994elliptic}, \cite{enol1994geometry} and \cite{maier2008}. 

We will end by saying that there is, as yet, no significant published
database of closed-form evaluations on Appell and Lauricella
functions, whether in print or on-line.  Volume 3 of Integrals and Series compendium, \cite{prudn}
and the Wolfram Functions site (functions.wolfram.com), both include extensive tables of
evaluations of univariate hypergeometric functions, but
multivariate ones remain to be treated.  Therefore, the
evaluations in this paper seem to break new ground.

%\subsubsection*{Acknowledgments}

%The authors wish to thank the anonymous referee whose valuable criticism improved deeply their article.

%\bibliographystyle{abbrv}
%\bibliography{kummer}

\begin{thebibliography}{10}

\bibitem{stegun}
M.~Abramowitz and I.~A. Stegun.
\newblock {\em Handbook of mathematical functions: with formulas, graphs, and
  mathematical tables}, volume~55.
\newblock Dover publications, 1965.

\bibitem{specfaar}
G.~Andrews, R.~Askey, and R.~Roy.
\newblock {\em Special Functions}.
\newblock Cambridge University Press, 1999.

\bibitem{ap}
P.~Appell and J.~Kamp{\'e}~de F{\'e}riet.
\newblock {\em Fonctions Hypergeometriques et Hyperspheriques; Polynomes
  d'Hermite}.
\newblock Gauthier-Villars, Paris, 1926.

\bibitem{bailey}
W.~Bailey.
\newblock {\em Generalized hypergeometric functions}.
\newblock Cambridge University Press, 1935.

\bibitem{Belokolos1986}
E.~D. Belokolos, A.~I. Bobenko, V.~B. Matveev, and V.~Enol'skii.
\newblock Algebro-geometric principles of superposition of ﬁnite-zone solutions
  of integrable nonlinear equations.
\newblock {\em Russian Math. Surveys}, 41(2):1--50, 1986.

\bibitem{BB}
J.~M. Borwein and P.~Borwein.
\newblock {\em Pi and the AGM: a study in the analytic number theory and
  computational complexity}.
\newblock Wiley-Interscience, 1987.

\bibitem{BZ}
J.~M. Borwein and I.~J. Zucker.
\newblock Fast evaluation of the gamma function for small rational fractions
  using complete elliptic integrals of the first kind.
\newblock {\em IMA journal of numerical analysis}, 12(4):519--526, 1992.

\bibitem{by}
P.~Byrd and M.~Friedman.
\newblock {\em Handbook of elliptic integrals for engineers and scientists}.
\newblock Springer Berlin, 1971.

\bibitem{carlson1976quadratic}
B.~Carlson.
\newblock Quadratic transformations of Appell functions.
\newblock {\em SIAM Journal on Mathematical Analysis}, 7(2):291--304, 1976.

\bibitem{Clebsch1865}
A.~Clebsch.
\newblock \"{U}ber diejenigen ebenen curven deren coordinaten rationale
  functionen eines parameters sind.
\newblock {\em J. Reine Angew. Math.}, 64:105--110, 1865.

\bibitem{eilbeck1994elliptic}
J.~Eilbeck and V.~Enol’skii.
\newblock Elliptic Baker--Akhiezer functions and an application to an
  integrable dynamical system.
\newblock {\em Journal of Mathematical Physics}, 35:1192, 1994.

\bibitem{enol1994geometry}
V.~Enol'skii and N.~Kostov.
\newblock On the geometry of elliptic solitons.
\newblock {\em Acta Applicandae Mathematicae}, 36(1-2):57--86, 1994.

\bibitem{mr00261160}
A.~Erd\'elyi.
\newblock Transformations of hypergeometric functions of two variables.
\newblock {\em Proceedings of the Royal Society of Edinburgh. Section A.
  Mathematical and Physical Sciences}, 62:378--385, 1948.

\bibitem{bat}
A.~Erd{\'e}lyi, W.~Magnus, F.~Oberhettinger, and F.~Tricomi.
\newblock {\em Tables of Integral Transforms}, volume~1.
\newblock McGraw-Hill New York, 1954.

\bibitem{bateman}
A.~Erd{\'e}lyi, W.~Magnus, F.~Oberhettinger, F.~G. Tricomi, and H.~Bateman.
\newblock {\em Higher transcendental functions}, volume~1.
\newblock McGraw-Hill New York, 1953.

\bibitem{goursat}
E.~Goursat.
\newblock Sur la r{\'e}duction des int{\'e}grales hyperelliptiques.
\newblock {\em Bulletin de la S. M. F.}, 13:143--162, 1885.

\bibitem{gra}
I.~Gradshtein, I.~Ryzhik, A.~Jeffrey, and D.~Zwillinger.
\newblock {\em Tables of integrals, series, and products}.
\newblock Academic press, 2007.

\bibitem{her}
C.~Hermite.
\newblock R$\acute{\mathrm{e}}$duction d'int$\acute{\mathrm{e}}$grales
  ab$\acute{\mathrm{e}}$liennes aux fonctions elliptiques.
\newblock {\em Annales de la soci$\acute{\mathrm{e}}$t$\grave{\mathrm{e}}$
  scientifique de Bruxelles}, 1:1--16, 1876.

\bibitem{herovres}
C.~Hermite.
\newblock {\em Oeuvres de Charles Hermite}, volume III.
\newblock Gauthier-Villars, 1912.

\bibitem{cre}
G.~Jacobi.
\newblock Nachschrift.
\newblock {\em J. Reine. Angew. Math.}, 8:416--417, 1832.

\bibitem{jz1}
G.~Joyce and I.~Zucker.
\newblock Special values of the hypergeometric series.
\newblock {\em Math. Proc. Camb. Phil. Soc}, 109:257--261, 1991.

\bibitem{jz3}
G.~Joyce and I.~Zucker.
\newblock Special values of the hypergeometric series iii.
\newblock {\em Math. Proc. Camb. Phil. Soc}, 133:213--222, 2002.

\bibitem{kummer}
E.~Kummer.
\newblock \"Uber die hypergeometrische reihe
  $1+\frac{\alpha\cdot\beta}{1\cdot\gamma}x+\frac{\alpha(\alpha+1)\cdot\beta(\beta+1)}{1\cdot2\cdot\gamma(\gamma+1)}x^2+\frac{\alpha(\alpha+1)(\alpha+2)\cdot\beta(\beta+1)(\beta+2)}{1\cdot2\cdot3\cdot\gamma(\gamma+1)(\gamma+2)}x^3\cdots$.
\newblock {\em J. Reine. Angew. Math.}, 15:39--83, 1836.

\bibitem{Lau}
G.~Lauricella.
\newblock Sulle funzioni ipergeometriche a pi\`{u} variabili.
\newblock {\em Rend. Circ. Mat. Palermo}, 7:111--158, 1893.

\bibitem{tr-1}
A.~M. Legendre.
\newblock {\em Trait{\'e} des fonctions elliptiques et des int{\'e}grales
  Euleriennes. Tome premier}.
\newblock Huzard-Courcier, Paris, 1825.

\bibitem{tr-2}
A.~M. Legendre.
\newblock {\em Trait{\'e} des fonctions elliptiques et des int{\'e}grales
  Euleriennes. Tome second}.
\newblock Huzard-Courcier, Paris, 1826.

\bibitem{maier2008}
R.~S. Maier.
\newblock Lam{\'e} polynomials, hyperelliptic reductions and Lam{\'e} band
  structure.
\newblock {\em Philosophical Transactions of the Royal Society A: Mathematical,
  Physical and Engineering Sciences}, 366(1867):1115--1153, 2008.

\bibitem{matsumoto2010transformation}
K.~Matsumoto.
\newblock A transformation formula for Appell's hypergeometric function $F_1$ and common limits of triple sequences by mean iterations.
\newblock {\em Tohoku Mathematical Journal}, 62(2):263--268, 2010.

\bibitem{matsumoto2009some}
K.~Matsumoto and K.~Ohara.
\newblock Some transformation formulas for Lauricella's hypergeometric
  functions $F_D$.
\newblock {\em Funkcialaj Ekvacioj}, 52(2):203--212, 2009.

\bibitem{jnt1}
G.~Mingari~Scarpello and D.~Ritelli.
\newblock The hyperelliptic integrals and $\pi$.
\newblock {\em Journal of Number Theory}, 129:3094--3108, 2009.

\bibitem{jnt2}
G.~Mingari~Scarpello and D.~Ritelli.
\newblock $\pi$ and the hypergeometric functions of complex argument.
\newblock {\em Journal of Number Theory}, 131:1887--1900, 2011.

\bibitem{prudn}
A.~P. Prudnikov, Y.~A. Brychkov, and O.~I. Marichev.
\newblock {\em Integrals and Series, Vol. 3: More Special Functions}.
\newblock Gordon and Breach, 1989.

\bibitem{R}
F.~Richelot.
\newblock \"Uber die reduction des integrals auf elliptische integrale.
\newblock {\em J. Reine. Angew. Math.}, 32:213--218, 1846.

\bibitem{Riemann1857}
B.~Riemann.
\newblock Lehrs\"{a}tze aus der analysis situs f\"{u}r die theorie der
  integrale von zweigliedrigen vollstandigen differentialien.
\newblock {\em J. Reine Angew. Math.}, 54:105--110, 1857.

\bibitem{SC}
A.~Selberg and S.~Chowla.
\newblock On Epstein's zeta function.
\newblock {\em J. Reine. Angew. Math.}, 227:88--110, 1967.

\bibitem{S}
J.~Serret.
\newblock {\em Cours de Calcul diff\'erentiel et int\'egral. Tome second}.
\newblock Gauthier Villars, Paris, 1886.

\bibitem{WW}
E.~T. Whittaker and G.~Watson.
\newblock {\em A course of modern analysis}.
\newblock Cambridge university press, 1996.

\bibitem{jz2}
I.~Zucker and G.~Joyce.
\newblock Special values of the hypergeometric series ii.
\newblock {\em Math. Proc. Camb. Phil. Soc}, 131:309--319, 2001.

\end{thebibliography}

%\end{document}

\end{document}